\documentclass[12pt]{amsart}
\usepackage{amsfonts}
\usepackage{amsmath}
\usepackage{amssymb}
\usepackage[margin=1in]{geometry}
\usepackage{graphicx}
\usepackage{hyperref}
\usepackage{xcolor}
\usepackage{mathtools}

\newcommand{\R}{\mathbb{R}}
\newcommand{\C}{\mathbb{C}}
\newcommand{\N}{\mathbb{N}}

\renewcommand{\L}{\mathcal{L}}
\renewcommand{\S}{\mathcal{S}}
\newcommand{\D}{\mathcal{D}}

\theoremstyle{plain}
\newtheorem{theorem}{Theorem}[section]
\newtheorem{lemma}[theorem]{Lemma}
\newtheorem{proposition}[theorem]{Proposition}
\newtheorem{corollary}[theorem]{Corollary}

\theoremstyle{definition}
\newtheorem{definition}[theorem]{Definition}

\newtheorem{remark}[theorem]{Remark}

\newtheorem{example}[theorem]{Example}

\numberwithin{equation}{section}
\DeclareMathOperator{\Real}{Re}

\DeclareMathOperator\supp{supp}
\DeclareMathOperator\singsupp{sing\; supp}

\DeclareMathOperator\rg{Rg}
\DeclareMathOperator\loc{loc}
\DeclareMathOperator\id{id}

\newcommand{\PF}{PF}

\newcommand{\uPF}{uPF}
\newcommand{\PFloc}{\PF_{\loc}}

\newcommand\ev[2]{\langle#1,#2\rangle}
\newcommand\wlim[1]{\text{w--} \lim_{#1}}

\begin{document}
\title[Spectral characterizations of stable semigroups]{Spectral characterizations of stable operator semigroups}

\author[M. Callewaert]{Morgan Callewaert}
\address{Department of Mathematics\\ Ghent University\\ Krijgslaan 297\\ B 9000 Gent\\ Belgium}
\email{morgan.callewaert@UGent.be}

\author[L. Neyt]{Lenny Neyt}
\address{University of Vienna\\ Faculty of Mathematics\\ Oskar-Morgenstern-Platz 1 \\ 1090 Wien\\ Austria}
\thanks{The research of L. Neyt was funded in whole by the Austrian Science Fund (FWF) 10.55776/ESP8128624. For open access purposes, the author has applied a CC BY public copyright license to any author-accepted manuscript version arising from this submission.}
\email{lenny.neyt@univie.ac.at}

\author[J. Vindas]{Jasson Vindas}
\thanks{The work of J. Vindas was supported by Ghent University through
the grant BOFBAF4y202401155 and by the Research Foundation--Flanders through the FWO-grant K801226N}
\address{Department of Mathematics\\ Ghent University\\ Krijgslaan 297\\ B 9000 Gent\\ Belgium}
\email{jasson.vindas@UGent.be}


\begin{abstract} 
We introduce the notion of local pseudofunction spectrum $\sigma_{\PF}(A)$ for the infinitesimal generator $A$ of a bounded $C_0$-semigroup $\mathcal{T} = (T(t))_{t \geq 0}$ on a Banach space $X$ and show it is the right spectral concept to deliver a full characterization of the strong stability of $\mathcal{T}$:
	\[ \forall x \in X : ~ \lim_{t \to \infty} \| T(t) x \|_X = 0  \quad \Longleftrightarrow \quad \sigma_{\PF}(A) = \varnothing. \]
We demonstrate how this yields a quick proof of the well-known Arendt-Batty-Lyubich-V\~u theorem and establish novel stability results for semigroups whose local pseudofunction spectra are contained in a 
null subset of the imaginary axis through localization of range density conditions to the potential null singular set.
We also obtain similar stability characterization theorems for individual orbits and for semi-uniform stability. As an application of our results, we provide spectral characterizations of asymptotically almost periodic $C_0$-semigroups with countable local pseudofunction spectrum.
In addition, we prove optimal Tauberian theorems of Katznelson-Tzafriri type and discuss connections with Wiener kernels.

\end{abstract}

\maketitle

\section{Introduction}\label{Introduction}

The stability theory of operator semigroups is a classical topic of functional analysis that has seen much progress in the past 50 years. The study of this subject has been greatly stimulated by its numerous applications to the theory of partial differential equations. We refer to the monographs \cite{A-B-H-N-VVLaplaceTransCauchyProb, vN-AsympBehSemigroupsLinOp} and the survey articles \cite{C-S-T-SemiUniformStabOpSemigroupEnergyDecay, C-T-IdeaResult} for excellent accounts on stability theory for $C_0$-semigroups on Banach spaces.

One of the most crucial and natural stability concepts for operator semigroups is that of \emph{strong stability}. A $C_0$-semigroup $\mathcal{T}=(T(t))_{t\geq0}$ on a Banach space $X$ is strongly stable if 
\[ \lim_{t\to\infty} \|T(t)x\|_{X}=0
\]
for each $x\in X$. 
Denoting $A$ as its infinitesimal generator, the mild solutions of the abstract Cauchy problem 
\begin{equation}
		\tag{ACP}
		\label{eq:ACP} 
		\begin{cases} 
			u'(t) = A u(t) , & t \geq 0 , \\  
			u(0) = x , & 
		\end{cases} 
	\end{equation}
are given by the trajectories $u(t)=T(t)x$. That the semigroup is strongly stable precisely means that every mild solution to \eqref{eq:ACP} vanishes at infinity.
In practice, the semigroup is rarely explicitly known, and one therefore turns to the rather delicate qualitative problem of determining when $\mathcal{T}$ is strongly stable from a priori available information on the operator $A$.

The central question from the point of view of both abstract operator theory and applications is then to \emph{characterize} strong stability in terms of spectral-like properties of the semigroup generator $A$, that is, using certain suitable information that may be extracted from its resolvent $R(\lambda, A)$. Except for the Hilbert space case \cite{T-ResApprStab}, no such full description has been known so far for $C_0$-semigroups on general Banach spaces. 
The aim of this article is to provide an answer to this important question via a new perspective that stems from relatively recent developments in complex Tauberian theory for Laplace transforms with so-called local pseudofunction boundary behavior \cite{D-V18, D-V-ComplexThmLaplaceLocalPseudo} (see also \cite{C-V26,T-V}). We shall, in fact, characterize strong stability in several ways using the boundary behavior of the resolvent $R(\lambda, A)$ on the imaginary axis. 

Perhaps one of the first and most celebrated criteria for strong stability is supplied by the \emph{ABLV theorem}, which was independently discovered by Arendt and Batty \cite{A-B-TaubThStabSemigr} and by Lyubich and V\~u \cite{L-V-AsympStabLinDiffEqBanach} in 1988. We write  $\sigma_{R}(A)$ for the residual spectrum of $A$.

	\begin{theorem}[ABLV theorem]
		\label{t:ABLV}
		Let $\mathcal{T}$ be a bounded $C_0$-semigroup.
		Suppose $\sigma(A) \cap i \R$ is countable and $\sigma_{R}(A)\cap i\R=\varnothing$.
		Then $\mathcal{T}$ is strongly stable.
	\end{theorem}
	
 Arendt and Batty argued \cite[Example 2.5(a)]{A-B-TaubThStabSemigr} that Theorem \ref{t:ABLV} is, in some sense, best possible.  Concretely, given an arbitrary closed uncountable set $E\subset \R$, they found an example of a unitary (hence non-stable) $C_0$-semigroup on a Hilbert space that satisfies $\sigma(A)\cap i\R\subseteq i E$ and $\sigma_{R}(A)\cap i\R=\varnothing$. 
 So, if one could still hope to find criteria for strong stability for semigroups with uncountable $\sigma(A)\cap i \R$, one might need to compensate for the increase of singularities by imposing stronger conditions on either the resolvent or on the ranges of $i \xi- A$ at points where singularities may occur. The following theorem is a well-known sample of such a strategy, where one basically requires the strongest possible range hypothesis in order to amend the absence of information on the spectrum.
 It has been observed in \cite{E-S-Z-StabAsympSemGroup} that the hypotheses of the ABLV theorem imply the intersection range density condition:
	\begin{equation}
		\label{eq: full range}
		 \bigcap_{\xi \in \R} \rg(i \xi - A) \text{ is dense in } X. 
	\end{equation}	
Batty, Chill, and Tomilov were able  \cite{B-C-T-StrongStabBoundEvFamSemigroups} (cf. \cite{T-ResApprStab}) to establish the sufficiency of \eqref{eq: full range} for strong stability.
	
	\begin{theorem}[{\cite[Theorem 6.3]{B-C-T-StrongStabBoundEvFamSemigroups}}]
		\label{t:OriginalRangeThm}
		Let $\mathcal{T}$ be a bounded $C_0$-semigroup.
		If the condition \eqref{eq: full range} is satisfied, then $\mathcal{T}$ is strongly stable.
	\end{theorem}

Milder variants of \eqref{eq: full range} have been extensively investigated in the semigroup stability literature \cite{B-C-T2007,C-T-StabilityGeometry,C-T2004}. The validity of stability criteria under such range intersection conditions turns out to be intimately tied to the geometry of the considered Banach spaces, and the study of this and related problems has led to deep developments on uniqueness theorems for boundary values of subharmonic functions \cite{B-C-T2007}. The next result \cite[Corollary 5.10]{B-C-T2007} generalizes Theorem \ref{t:OriginalRangeThm}:

	\begin{theorem}[{\cite[Corollary 5.10]{B-C-T2007}}]
		\label{t:FourierTypeRangeThm}
		Let $\mathcal{T}$ be a bounded $C_0$-semigroup on a Banach space of Fourier type $p\in[1,2]$. If 
			\begin{equation}
				\label{eq:gen range}
		 		\bigcap_{\xi \in \R} \rg(i \xi - A)^{\gamma} \text{ is dense in } X
			\end{equation}			
		with (sharp) exponent $\gamma=1/p$, then $\mathcal{T}$ is strongly stable.
	\end{theorem}

Both Theorems \ref{t:ABLV} and \ref{t:FourierTypeRangeThm} do not characterize strongly stable $C_0$-semigroups. 
For the former case, one can find examples of strongly stable $C_0$-semigroup with $ \sigma(A)\cap i\R=i\R$ in  \cite[Example 4.1]{B-V-R-LocalSpecInStab} and \cite[Example 2.5(c)]{A-B-TaubThStabSemigr}, while the fact that the range condition \eqref{eq:gen range} is not necessary for strong stability was demonstrated in \cite[Section 4]{C-T-StabilityGeometry}. 
 There are several other useful pointwise and integrability conditions that are either necessary or sufficient for strong stability and cover the range intersection conditions \eqref{eq:gen range}; see, for instance, \cite{B-C-T2007,C-T-StabilityGeometry, C-T2004,C-T-IdeaResult, T-ResApprStab}. Some of them are known to characterize strong stability in the case of Hilbert spaces but not for general Banach spaces. 
 We also refer to \cite{G-G-ConvPositiveOpSemigroups} for the study of strongly convergent positive operator semigroups.
	
In this work, we shall introduce a new type of spectrum for the infinitesimal generator $A$ of a bounded $C_0$-semigroup $\mathcal{T}$, which we call the \emph{local pseudofunction spectrum} and denote as $\sigma_{\PF}(A)$. The closed set $\sigma_{\PF}(A)$ is always contained in $\sigma(A) \cap i \R$ and is given by the complement of the largest open subset of $i\R$ where the function $R(\lambda, A) x$ has so-called ($X$-valued) local pseudofunction boundary behavior for each $x \in X$ (see Definition \ref{def:LocalPseudofuncSpect} and \eqref{eq:UnionIsFullPseudoSpectrum} 
for details). The local pseudofunction spectrum will be the key to our spectral characterization of strong stability:

	\begin{theorem}
		\label{t:StrongStabCharIntro}
		Let $\mathcal{T}$ be a
		$C_0$-semigroup. The following statements are equivalent:
		\begin{itemize}
		\item [$(a)$] $\mathcal{T}$ is strongly stable;
		\item [$(b)$] $\mathcal{T}$ is bounded and $\sigma_{PF}(A)= \varnothing$;
		\item [$(c)$] $\mathcal{T}$ is bounded, $\sigma_{PF}(A)$ is countable, and $\sigma_{R}(A)\cap i\R=\varnothing$;
		\item [$(d)$] $\mathcal{T}$ is bounded, $i\sigma_{\PF}(A)$ is a null subset of the real line, and
		\begin{equation}
		\label{eq: local range}
		 \bigcap_{i\xi \in \sigma_{\PF}(A)\cap iI} \rg(i \xi - A) \text{ is dense in } X 
		\end{equation}
			for each bounded interval $I\subset \R$.
				\end{itemize}
	\end{theorem}
We point out that all of the properties $(b)$--$(d)$ are purely expressible in terms of the resolvent of the semigroup. Indeed, by the Hille-Yosida-Feller-Miyadera-Phillips generation theorem \cite[p.~69]{E-N-ShortCourseSemigroups}, the semigroup $\mathcal{T}$ is bounded if and only if $\sigma(A) \subseteq \{ \lambda \in \C \mid \Real{\lambda} \leq 0\}$ and its resolvent satisfies  the bounds
	\[  \|R(\lambda, A)^n\|_{L(X)} \leq \frac{C}{(\Real{\lambda})^n}, \qquad \Real{\lambda} > 0, ~ n \in \N , \]
for some constant $C>0$.

 Our proof of Theorem \ref{t:StrongStabCharIntro} will be based on vector-valued versions of scalar-valued Tauberian theorems due to Debruyne and the last-named author \cite{D-V-ComplexThmLaplaceLocalPseudo}, which take boundary requirements on Laplace transforms in the Ingham-Karamata theorem to a minimum and have served as an inspiration for our definition of the local pseudofunction spectrum. We mention that the use of pseudofunctions in Tauberian theory was initiated by Katznelson and Tzafriri \cite{K-T-PowerBoundOp} in the case of power series, and later by the pioneering work of Korevaar \cite{K2005,K-TaubTh} on complex Tauberian theorems for Laplace transforms. It should also be pointed out that these ideas have already found important applications in abstract analytic number theory, where the explicit determination of the local pseudofunction boundary behavior of expressions involving Beurling zeta functions and their derivatives often becomes crucial to reach many sharp asymptotic results on the behavior of generalized primes and integers; see, for instance, \cite{D-V17,D-V19,T-V}.

 The equivalences $(a)\iff (c)\iff (d)$ from Theorem \ref{t:StrongStabCharIntro} significantly improve upon the ABLV theorem. We provide in Section \ref{section lpfs} examples highlighting that in general $\sigma_{\PF}(A)$ can be a much smaller set than $\sigma(A)\cap i\R$. Furthermore, as will be illustrated in Section \ref{sec:TransportEq}, familiar PDEs such as the transport equation
 	\begin{equation} 
		\label{intro-eq:TransportEq}
		\frac{\partial u}{\partial t} + a(s)\frac{\partial u}{\partial s} = 0 , \qquad (t, s) \in \R_+ \times \R_+ ,
	\end{equation}
when considered on appropriate Banach function spaces, lead to examples of Cauchy problems whose strong stability can be verified through our local pseudofunction spectral characterization Theorem \ref{t:StrongStabCharIntro}, but not via classical results such as Theorems \ref{t:ABLV} or \ref{t:FourierTypeRangeThm}.
 
 Property \eqref{eq: local range} should be regarded as a local refinement of the full range condition \eqref{eq: full range}. While most previous range intersection stability criteria exploit the dual complete trajectory characterization of strong stability in combination with edge-of-the-wedge theorems, the novelty in our approach consists in the localization of the range condition to the potential local pseudofunction null exceptional set, which, together with optimal Tauberian theorems, allows us to remove the local pseudofunction boundary singularities for the resolvent and in this way to demonstrate the desired equivalence with the strong stability of the semigroup.

It is important to observe that property $(d)$ is always applicable in situations when analog information on the usual spectrum $\sigma(A)$ is available. As a matter of fact, Theorem \ref{t:StrongStabCharIntro} tells that it is possible to establish a new criterion for strong stability based on the size of $\sigma(A)\cap i \R$ even when this set is uncountable, if this hypothesis is augmented with a counterpart of the local range condition \eqref{eq: local range} for $\sigma(A)$:

\begin{corollary}
\label{c: ss}
Let $\mathcal{T}$ be a bounded
		$C_0$-semigroup. If $\sigma(A)\cap i\R$ is a null set and
		\begin{equation}
		\label{eq: local range ibs}
		 \bigcap_{i\xi \in \sigma(A)\cap iI} \rg(i \xi - A) \text{ is dense in } X 
		\end{equation}
		for each bounded interval $I\subset \R$, then $\mathcal{T}$ is strongly stable.
\end{corollary}

In \cite[Theorem 3.4]{B-V-R-LocalSpecInStab}, Batty, van Neerven, and R\"abiger showed an extension of the ABLV theorem that provides sufficient conditions for stability of individual orbits in terms of the so-called local unitary spectrum $\sigma_{u}(A, x)\subseteq i\R$ at a vector $x\in X$. We shall also introduce the local pseudofunction spectrum $\sigma_{\PF}(A,x)\subseteq i\R$ at $x\in X$ and show that $\sigma_{\PF}(A,x)$  is the right spectral concept to deliver a full characterization of the stability of an orbit, as stated in the next theorem that (strictly) contains \cite[Theorem 3.4]{B-V-R-LocalSpecInStab}. Observe that, in view of the mean ergodic theorem \cite{A-B-H-N-VVLaplaceTransCauchyProb}, the condition \eqref{eq: mean ergodic intro} below is equivalent to  $x \in \overline{\rg{(i \xi - A)}}$.

	\begin{theorem}
		\label{t:OrbitStabIntro}
	Let $\mathcal{T}$ be a bounded $C_0$-semigroup and let $x \in X$. The following properties are equivalent:
	\begin{itemize}
	\item [$(a)$] $\lim_{t\to\infty}\|T(t)x\|_{X}=0$;
	\item [$(b)$] $\sigma_{\PF}(A,x)=\varnothing$;
	\item [$(c)$] $\sigma_{\PF}(A,x)$ is countable and, for each  $i\xi \in \sigma_{\PF}(A,x)$,
\begin{equation} \label{eq: mean ergodic intro}
\lim_{\alpha \to 0^+}\alpha \|R(\alpha + i \xi, A) x\|_X = 0
	.
\end{equation}
	\end{itemize}
	\end{theorem}
	
As an application of Theorem \ref{t:OrbitStabIntro}, we will derive the next description of relatively compact orbits, which in particular yields a spectral characterization of asymptotically almost periodic semigroups, under the assumption that $\sigma_{\PF}(A)$ is countable.

\begin{theorem}
\label{t:CharAlmostPeriodic}
Let $\mathcal{T}$ be a bounded $C_0$-semigroup and let $x\in X$ be such that $\sigma_{\PF}(A, x)$ is countable. The orbit $\{T(t)x|\:  t\geq 0\}$ is relatively compact if and only if $\lim_{\alpha \to 0^+} \alpha R(\alpha + i \xi, A)x$ exists for each $i\xi \in \sigma_{\PF}(A,x)$.
\end{theorem}

In their seminal work, Katznelson and Tzafriri connected the asymptotic theory of power bounded operators with spectral synthesis on the unit circle. Spectral synthesis is a concept of much relevance in harmonic analysis; see \cite{benedetto,K-IntroHarmAnal}. In the case of the torus $\mathbb{T}$, given a power series $F(z)=\sum_{n=0}^{\infty} a_n z^n$ with absolutely convergent coefficients $(a_n)\in \ell^{1}$, one says that it is of spectral synthesis with respect to a closed subset $E\subseteq \mathbb{T}$ if the Fourier series $F(e^{i\theta})$ is of spectral synthesis with respect to it. An operator $T$ on a Banach space $X$ is called power bounded if $\sup_{n\in \N}\|T^{n}\|_{L(X)}<\infty$.

\begin{theorem}[{\cite[Theorem 5 \&  comments on p.~319]{K-T-PowerBoundOp}}] \label{t:K-Tintro} Let $E\subset \mathbb{T}$ be a closed null set. A power series $F$ with absolutely convergent coefficients is of spectral synthesis with respect to $E$ if and only if $\lim_{n\to\infty} \|T^{n}F(T)\|_{L(X)}=0$ for all power-bounded operators $T$ with $\sigma(T)\cap \mathbb{T}\subseteq E$.
\end{theorem}

Generalizations of the Katznelson-Tzafriri theorem to the one-parameter continuous case were established in \cite{B-V-R-LocalSpecInStab,E-S-Z-StabAsympSemGroup, V-ThmKatznelsonTzafriri}.
In Section \ref{sec:KatznelsonTzafriri}, we will obtain optimal forms of the Katznelson-Tzafriri theorem on the real half-line that sharpen those results and that are the complete analogs of Theorem \ref{t:K-Tintro} for the three spectra $\sigma(A)$, $\sigma_{\PF}(A)$, and $\sigma_{\PF}(A,x)$. We only state here the case involving the local pseudofunction spectrum, and refer to Theorem \ref{t:KatznelsonTzafriri} below for the other two cases.

\begin{theorem}
	\label{t:KatznelsonTzafririIntro}
	Let $f \in L^1(\R_+)$ and let $E\subset\R$ be a closed null set.
	The function $f$ is of spectral synthesis with respect to $E$ if and only if for any bounded $C_0$-semigroup $\mathcal{T}$ the following property holds:
		\begin{equation*}
		\label{eq:KatznelsonTzafririIntro} 
		i \sigma_{\PF}(A) \subseteq E \quad \Longrightarrow \quad \lim_{t \to \infty} \Big\| \int_0^\infty f(s)T(t + s) x \,\mathrm{d}s  \Big\|_{X} = 0,  \quad \forall x \in X . 
	\end{equation*}
\end{theorem}

The article is organized as follows.  We discuss some preliminary material on local pseudofunctions and (distributional) boundary values of vector-valued analytic functions in Section \ref{Pseudo}. We provide there Banach space-valued versions of several results from \cite{D-V-ComplexThmLaplaceLocalPseudo}; in particular, Theorems \ref{t:SuffCondVanishSingSupp} and \ref{MainResultGregoryJasson} will serve as the primary Tauberian tools in our analysis. The proof of Theorem \ref{t:SuffCondVanishSingSupp} is postponed to Appendix \ref{appendix:ProofTheorem}.
In Section \ref{section lpfs}, we introduce and study the local pseudofunction spectra $\sigma_{\PF}(A) $ and $\sigma_{\PF}(A,x) $ at $x\in X$ for the infinitesimal generator $A$ of a bounded $C_0$-semigroup. We supply some examples that clarify how these local pseudofunction spectra may drastically differ from $\sigma(A)\cap i\R$ and the local unitary spectrum $\sigma_{u}(A,x)$ at $x\in X$. Section \ref{sec:StabOrbits} is devoted to proving Theorem \ref {t:OrbitStabIntro}, while we establish a refinement of Theorem \ref{t:StrongStabCharIntro} in Section \ref{sec:StrongStabilityChar}. 
To illustrate the applicability of our theory in concrete settings, we use our results in Section \ref{sec:TransportEq} to establish the strong stability of solutions to the transport equation \eqref{intro-eq:TransportEq} on weighted $L^1$-spaces on $\R_{+}$ by direct computation of the local pseudofunction spectrum of the associated infinitesimal generator.
We also show that, for these examples, such stability cannot generally be deduced from Theorems \ref{t:ABLV} or \ref{t:FourierTypeRangeThm}.
We take a detour and consider semi-uniformly stable $C_0$-semigroups \cite{B-D-NonUniformStab,C-S-T-SemiUniformStabOpSemigroupEnergyDecay} in Section \ref{sec:SemiUniform}. Since the resolvent $R(\lambda, A)$ has $L(X)$-valued distributional boundary values on the imaginary axis (see Subsection \ref{sub: bv and taub} and Section \ref{section lpfs}), it is natural to define a uniform version of $\sigma_{\PF}(A)$, the uniform local pseudofunction spectrum denoted by $\sigma_{\uPF}(A)$, as the complement of the largest open subset of $i\R$ where $R(\lambda, A)$ has $L(X)$-valued local pseudofunction boundary behavior. We prove in Section \ref{sec:SemiUniform} that actually $\sigma_{\uPF}(A)=\sigma(A)\cap i\R$ and recover the following result \cite{B-D-NonUniformStab}: a bounded $C_0$ semigroup is semi-uniformly stable if and only if $\sigma(A)\cap i\R=\varnothing$. 
In Section \ref{sec:KatznelsonTzafriri}, we show our optimal version of the Katznelson-Tzafriri theorem on the real half-line, which contains Theorem \ref{t:KatznelsonTzafririIntro}. In addition, we study connections between our Katznelson-Tzafriri theorem and Wiener kernels in Section \ref{Wiener kernels}. The final Section \ref{sec:AlmostPeriodic} deals with our characterization of asymptotically almost periodic semigroups, and in particular we prove Theorem \ref{t:CharAlmostPeriodic} there.

Finally, we conclude this introduction by mentioning that the scope of this paper is the study of the \emph{qualitative} decay of $C_0$-semigroups, which, as explained above, we demonstrate to be fully determined by the local pseudofunction spectrum of the generator. There is, on the other hand, a substantial body of literature dealing with quantitative aspects of stability theory, where quantified decay rates may be deduced if one requires stronger regularity and growth bound hypotheses on the resolvent. The interested reader can consult \cite{B-C-T2016, B-D-NonUniformStab, B-T-OptPolyDecayFuncOpSemigroups,C-N-V-OptDecaySemiUniformStableOpSemigroups, C-S-QuantIngham, C-S-T-SemiUniformStabOpSemigroupEnergyDecay, D-S-AbstractOptDecayFuncOpSemigroups, D-S-OptQuantifiedInghamKaramata, R-S-S-OptRatesDecayOpSemigroupsHilbert, S-C24, S-DecayC0Semigroups} for this topic. We also bring the reader's attention to the recent reference \cite{d26}, which develops a general remainder Tauberian theory for Laplace transforms.

\section{Vector-valued local pseudofunctions and Tauberian theorems}\label{Pseudo}
In this section, we collect some background material on local pseudofunctions and local pseudofunction boundary behavior of analytic functions in the vector-valued setting and explain their connection with Tauberian theorems. In particular,  we introduce the \emph{local pseudofunction singular support} of vector-valued distributions, and show (Lemma \ref{l:SingSuppEmpty<=>o(1)}) how having empty local pseudofunction singular support allows one to deduce that a vector-valued function vanishes at infinity if it satisfies natural Tauberian conditions. We then present two situations in which being empty for the singular support follows from being contained in a null set together with some extra local Tauberian conditions (Theorems \ref{t:SuffCondVanishSingSupp} and \ref{MainResultGregoryJasson}). These results form the basis for our analysis of stability in the next sections and may be viewed as the vector-valued adaptations of the scalar-valued Tauberian results obtained by Debruyne and the last-named author in \cite{D-V-ComplexThmLaplaceLocalPseudo}; some details on the proofs of these vector-valued adaptations are given in Appendix \ref{appendix:ProofTheorem}.

\subsection{Vector-valued function and distribution spaces}
Given two topological vector spaces $V$ and $W$, we denote by $L(V, W)$ the space of all continuous linear maps $V \to W$. As customary, we simply write $L(V)=L(V,V)$.

Let $Z$ be a Banach space---an assumption we will make throughout this section. We denote as $C_0(\R; Z)$  the space of all continuous $Z$-valued functions that vanish at infinity.

For any open set $I \subseteq \R$, the space of $Z$-valued distributions on $I$ is given by $\D^\prime(I; Z) = L(\D(I), Z)$, where $\D(I)$ denotes the space of smooth functions with compact support in $I$, endowed with its natural strict $(LF)$-space topology \cite{S-ThDist}.
We also consider the space $\S^\prime(\R; Z) = L(\S(\R), Z)$ of $Z$-valued tempered distributions, where
	\[ \S(\R) = \{ \varphi \in C^\infty(\R) \mid \sup_{t \in \R} (1 + |{t}|)^k |\varphi^{(n)}({t})| < \infty \text{ for all } n, k \in \N \} \]
is the Schwartz space of rapidly decreasing smooth functions \cite{S-ThDist}, which carries its canonical Fr\'echet space topology. We write $\langle f,\varphi\rangle \in Z$, or $\langle f(t),\varphi(t)\rangle \in Z$ with the use of a dummy variable of evaluation, for the evaluation of a $Z$-valued distribution $f$  at a test function $\varphi$. Accordingly, sometimes we also write expressions like $f(t)\in \mathcal{D}'(I;Z)$, to be interpreted as $f\in \mathcal{D}'(I;Z)$. As usual, $Z$-valued locally integrable functions $f$ are regarded as $Z$-valued distributions via $\langle f(t),\varphi(t)\rangle=\int_{-\infty}^{\infty}\varphi(t) f(t)\, \mathrm{d}t$.

We fix the constants in the Fourier transforms as 
$$\widehat{\varphi}(\xi)=\mathcal{F}\{\varphi;\xi\}=\int_{-\infty}^{\infty}e^{-i\xi t}\varphi(t)\:\mathrm{d}t.
$$ 
Since the Fourier transform $\mathcal{F}:\mathcal{S}(\R)\to \mathcal{S}(\R)$ is a Fr\'{e}chet space automorphism, we can extend it to $f\in\mathcal{S}'(\mathbb{R};Z)$ by transposition, that is, if $f\in\mathcal{S}'(\mathbb{R}; Z)$, its Fourier transform is the tempered distribution $\widehat{f}\in\mathcal{S}'(\mathbb{R};Z)$ determined by $\langle \widehat{f}(\xi),\varphi(\xi)\rangle=\langle f(t), \widehat\varphi(t)\rangle\in Z$ for all test functions $\varphi\in \mathcal{S}(\mathbb{R})$.

\subsection{Vector-valued local pseudofunctions}\label{sub v-v pseudofunctions}
We now introduce the following notions.

\begin{definition}
Let $I \subseteq \R$ be open.

\begin{itemize}
\item[$(i)$] The space of ($Z$-valued global) pseudofunctions  is defined as $\PF(\mathbb{R};Z)=\mathcal{F}(C_{0}(\R;Z))$.

\item[$(ii)$] Given a point $\xi_0 \in I$, we say $f \in \D^\prime(I; Z)$ is a \emph{local pseudofunction at $\xi_0$} and write $f\in \PFloc(\xi_{0};Z)$ if $f$ coincides with a pseudofunction on some open neighborhood $I' \subseteq I$ of $\xi_0 \in I'$.

\item[$(iii)$] We write
	\[ \PFloc(I; Z) = \{ f \in \D^\prime(I; Z) \mid f\in \PF(\xi_0;Z) \text{ for each } \xi_0 \in I \}, \]
for the space of local pseudofunctions on $I$.
\end{itemize}
\end{definition}

Since obviously $\mathcal{S}(\R)\ast C_{0}(\R; Z)\subset C_{0}(\R;Z)$, the space of pseudofunctions  becomes a multiplication module over $\mathcal{S}(\R)$. This clearly induces a multiplication module structure on $\PFloc(I; Z)$ over the algebra $C^{\infty}(I; Z)$. Just as $\mathcal{D}'(\,\cdot\,; Z)$, it is then very important to point out that such a multiplication and the existence of $C^{\infty}$-partitions of unity yield that $\PFloc(\,\cdot\,; Z)$ is a fine sheaf of vector spaces (the sheaf restriction mappings being given by restriction in the sense of distributions). In particular, 
a given distribution $f \in \D^\prime(I; Z)$ is a local pseudofunction on the open subset $I \subseteq \R$ if and only if, for any $\varphi \in \D(I)$,
	\begin{equation}
		\label{eq:LocalPseudoTestFunc}
		\| \ev{f(\xi)}{e^{it\xi} \varphi(\xi)} \|_Z = o(1) , \qquad |t| \to \infty . 
	\end{equation}
Hence, $L^1_{\loc}(I; Z) \subset \PFloc(I; Z)$, with strict inclusion (see Lemma \ref{l:SuffCondNonContFourier} and Example \ref{ex:PFEmptyButNotContinuous}), in view of the classical Riemann-Lebesgue lemma.

\begin{definition}
Let $I\subseteq \R$ be open. The local pseudofunction singular support of $f \in \D^\prime(\R; Z) $ is the closed set defined as
	\[ 
	\singsupp_{\PF} f = \R \setminus \{ \xi_0 \in \R \mid f \in\PFloc(\xi_0;Z) \} .  
	\]
\end{definition}

Clearly, the local pseudofunction singular support is the complement of the largest open subset where $f$ is a local pseudofunction.

When imposing specific Tauberian conditions on a function $\tau$, the assumption 
	\[ \singsupp_{\PF} \widehat{\tau} =\varnothing \] 
can give information on the asymptotic behavior of $\tau$. The latter problem has been extensively studied in \cite{D-V-ComplexThmLaplaceLocalPseudo} (in the scalar-valued case).
We will consider here the classical Tauberian condition of slow oscillation \cite{K-TaubTh}.
The function $\tau$ is \emph{slowly oscillating} (at $\pm\infty$) if
	\begin{equation}\label{eqso} \lim_{\delta \to 0+} \limsup_{|t| \to \infty} \sup_{s\in[0,\delta]} \|\tau(t+s)-\tau(t)\|_Z = 0 . \end{equation}

\begin{lemma}
\label{l:SingSuppEmpty<=>o(1)}
Suppose $\tau \in L^1_{\loc}(\R; Z)$ is slowly oscillating.
Then, $\singsupp_{\PF} \widehat{\tau} =\varnothing$ if and only if
	\[ \tau(t) = o(1) , \qquad \text{as } |t| \to \infty . \]
\end{lemma}

\begin{proof}
We first point out that the slow oscillation yields $\tau(t)=O(t)$ (see a finer inequality below), so that we get $\tau\in \mathcal{S}'(\R; Z)$.

Suppose $\tau(t) = o(1)$ as $|t| \to \infty$.
Then $\phi * \tau \in C_0(\R; Z)$ for any $\phi \in \S(\R)$.  
Taking the Fourier transform shows that $\widehat{\tau} \in \PFloc(\R; Z)$, that is, $\singsupp_{\PF} \widehat{\tau} =\varnothing$.  

Let us now assume that $\singsupp_{\PF} \widehat{\tau} =\varnothing$.
We claim that, for any $\phi \in \S(\R)$, 
	\begin{equation}
		\label{eq:Conv-o(1)} 
		(\phi * \tau )(t) = o(1) ,\qquad \text{as } |t| \to \infty  . 
	\end{equation} 
Before proving \eqref{eq:Conv-o(1)}, let us first demonstrate how it yields $\tau(t)=o(1)$. Take any $\varepsilon > 0$. 
There exist $t_0>0$ and $\delta>0$ for which $\|\tau(t + s) - \tau(t)\|_{Z} \leq \varepsilon$ whenever $|s| \leq \delta$ and $|t| \geq t_0$.
Fix now a non-negative $\phi \in \D(\R)$ such that $\supp \phi \subseteq [0, \delta]$ and $\int_{-\infty}^{\infty} \phi(s)\,\mathrm{d}s = 1$.
Then, using \eqref{eq:Conv-o(1)} and writing
 \begin{equation}\label{eq:formulaTau} \tau(t) = \int_{-\infty}^\infty \phi(s)\tau(t)  \, \mathrm{d}s = (\phi \ast \tau)(t) + \int_{-\infty}^\infty \phi(s)(\tau(t)-\tau(t - s)) \, \mathrm{d}s, 
 \end{equation}
we obtain
	\[ \limsup_{|t| \to \infty} \|\tau(t)\|_Z \leq \limsup_{|t| \to \infty} \Big\| \int_{0}^{\delta} \phi(s)(\tau(t - s) - \tau(t)) \, \mathrm{d}s \Big\|_Z \leq \varepsilon . \]
As $\varepsilon > 0$ was arbitrary, this would complete the proof.
	
Let us now return to \eqref{eq:Conv-o(1)}.
First, for any $\phi = \widehat{\varphi}$ with $\varphi \in \D(\R)$, we have, by \eqref{eq:LocalPseudoTestFunc}, 
	\[ (\phi * \tau) (t) = \ev{\widehat{\tau}(\xi)}{e^{it\xi} \varphi(-\xi)} = o(1) , \qquad \text{as } |t| \to \infty . \] 
So, \eqref{eq:Conv-o(1)} holds for all $\phi \in \mathcal{F}(\mathcal{D}(\R))$. Notice that $\tau(t) = O(1)$ as $|t| \to \infty$. Indeed, 
by iterating the resulting inequality from \eqref{eqso}  and potentially modifying $\tau$ on a bounded interval, we may assume that there is some $M > 0$ such that $\|\tau(t+s)-\tau(t)\|_Z \leq M(|s| + 1)$ for all $t,s \in \R$.
Let $\phi \in \mathcal{F}(\D(\R))$ be such that $\phi\geq0$ and $\int_{-\infty}^{\infty} \phi(s)\,\mathrm{d}s = 1$. Filling this in \eqref{eq:formulaTau}, we have
\[\|\tau(t)\|_{Z}\leq o(1)+M\int_{-\infty}^{\infty}(|s|+1)|\phi(s)|\mathrm{d}s=O(1). \]
Using now the fact that $\mathcal{F}(\D(\R))$ is dense in $\S(\R)$, we can then conclude that \eqref{eq:Conv-o(1)} holds for each $\varphi\in \mathcal{S}(\R)$.
\end{proof}

We now consider a sufficient condition on $\singsupp_{\PF} \widehat{\tau}$ that guarantees its emptiness.
In fact, this has already been done in \cite{D-V-ComplexThmLaplaceLocalPseudo} for the scalar-valued case in more generality. For the purposes of this article, we will only need the following vector-valued version of \cite[Corollary 4.3]{D-V-ComplexThmLaplaceLocalPseudo}, whose proof is presented in Appendix \ref{appendix:ProofTheorem} for the sake of completeness.

\begin{theorem}
\label{t:SuffCondVanishSingSupp}
   Let $\tau\in L^{\infty}(\R; Z) $ and let $I\subseteq \R$ be open.
   Suppose there is a null subset $E\subset I$ for which $I\cap\singsupp_{\PF} \widehat{\tau} \subseteq E$ and such that
        \begin{equation} \label{BoundFourierTransform2}
        \sup_{t \in\R}\Big\Vert \int^t_0 e^{-i \xi_0 s} \tau(s) \, \mathrm{d}s \Big\Vert_Z  < \infty , \qquad \forall \xi_0 \in E .
    \end{equation}
Then, $I\cap \singsupp_{\PF} \widehat{\tau} = \varnothing$.
\end{theorem}

\subsection{Vector-valued Tauberian theorem for local pseudofunction boundary behavior}
\label{sub: bv and taub}
Let $F : \C_+ \to Z$ be analytic, where $\C_+ = \{ \lambda \in \C \mid \Real \lambda > 0 \}$.
For any open $I \subseteq \R$, we say that $F$ has \emph{distributional boundary values} on the boundary set $i I$ if there exists some $f \in \D^\prime(I; Z)$ such that 
	\begin{equation}
		\label{eq:BoundaryBehavior} 
		\lim_{\alpha \to 0^+} \int_{-\infty}^\infty \varphi(\xi)F(\alpha + i\xi)  \, \mathrm{d}\xi=\ev{f(\xi)}{\varphi(\xi)} , \qquad \varphi \in \D(I) . 
	\end{equation}
We simply write $F(i\xi) = f(\xi)$ for the distributional boundary value (or boundary distribution) of $F$, which should always be interpreted as an element of $\D'(I; Z)$.
Having distributional boundary values is equivalent to the following local boundary bounds \cite[Theorem 3.1.11, p.~63]{H-PDOI}: For a fixed $\alpha_0 > 0$ and for each bounded open $I' \subseteq I$ there is $N = N_{I'}$ such that
	\begin{equation}
		\label{BoundaryValuesBound}
		\|  F(\alpha + i\xi)\|_{Z} = O\left(\frac{1}{\alpha^N}\right), \qquad  \mbox{for } \alpha + i\xi \in (0,\alpha_0] + i I' .
	\end{equation}

\begin{definition}
\label{def:LocalPseudofuncBoundBeh}
Let $I \subseteq \R$ be open.
We say that a holomorphic function $F : \C_+ \to Z$ has \emph{local pseudofunction boundary behavior} on $i I$ if \eqref{eq:BoundaryBehavior} holds for some $f \in \PFloc(I; Z)$.
\end{definition}

Suppose $\tau : \R_{+}\to Z$ is a polynomially bounded (Bochner) measurable function, where here and below we use the notation $\R_{+}=[0,\infty)$.
Then its Laplace transform
	\[ \L\{\tau; \lambda\} = \int_{0}^\infty e^{-\lambda t}\tau(t)  \, \mathrm{d}t , \qquad \lambda \in \C_+ , \]
defines a holomorphic function $\C_+ \to Z$.
Moreover, one readily shows the global bound, for some $N> 0$,
\[ 
\|\L\{\tau; \lambda\} \|_{Z} = O\left(\frac{1}{\alpha^{N}}\right) , \qquad \lambda = \alpha + i \xi \in \C_+ . \]
Consequently, $\L\{\tau; \lambda\}$ has distributional boundary value $\L\{\tau; i\xi\} \in \D^\prime(\R; Z)$ on the whole imaginary axis.
In fact, its boundary distribution coincides with the Fourier transform of $\tau$
\cite{B-DistrComplVarFourTrans,V-MethGenlFunct} 
	\begin{equation}
		\label{eq:FourierTransformBoundaryLaplace} 
		\widehat{\tau}(\xi) = \L\{\tau; i\xi\} = \lim_{\alpha \to 0^+} \L\{\tau; \alpha + i\xi\} , 
	\end{equation}
and the limit actually holds in $\S'(\R; Z)$.

We now find the following special version of the Tauberian theorem \cite[Theorem 5.1]{D-V-ComplexThmLaplaceLocalPseudo} in the vector-valued case.

\begin{theorem}\label{MainResultGregoryJasson}
Let $\tau: \R_{+}\to Z$ be bounded and uniformly continuous. The following statements are equivalent:
	\begin{itemize}
		\item[$(a)$] $\tau(t) = o(1)$ as $t \to \infty$;
		\item[$(b)$]  $\mathcal{L}\{\tau; \lambda\}$ has local pseudofunction boundary behavior on $i\R$;
		\item[$(c)$] there is a closed null set $E \subset \R$ such that $\mathcal{L}\{\tau; \lambda\}$ has local pseudofunction boundary behavior on $i(\R\setminus E)$ and
    			\begin{equation*} 
        				\sup_{t\in\R_{+}} \Big\Vert \int^t_0 e^{-i \xi_0 s}\tau(s) \, \mathrm{d}s \Big\Vert_Z < \infty , \qquad \forall \xi_0 \in E .
    			\end{equation*}
	\end{itemize}
\end{theorem}

\begin{proof}By \eqref{eq:FourierTransformBoundaryLaplace} we have $\singsupp_{\PF} \L\{\tau; i \, \cdot\} = \singsupp_{\PF} \widehat{\tau}.$ The property $(b)$ is thus the same as $\singsupp_{\PF} \widehat{\tau}=\varnothing,$ while the first part of $(c)$ means that $\singsupp_{\PF} \widehat{\tau}\subseteq E.$
The equivalence $(a) \Longleftrightarrow (b)$ is now a special case of Lemma \ref{l:SingSuppEmpty<=>o(1)}.
The implication $(b) \Longrightarrow (c)$ is trivial.
Finally, $(c) \Longrightarrow (b)$ follows from Theorem \ref{t:SuffCondVanishSingSupp}.
\end{proof}

\section{The local pseudofunction spectrum and bounded $C_0$-semigroups}\label{section lpfs}
In this section, we introduce and study the concept of \emph{local pseudofunction spectrum} $\sigma_{\PF}(A)$ for the infinitesimal generator $A$ of a bounded strongly continuous semigroup. 

Recall a $C_0$-semigroup $\mathcal{T} = (T(t))_{t \geq 0}$ is an operator-valued map $T:\R_{+}\to L(X)$, where $X$ is a Banach space, satisfying
	\[ T(0) = \id_X , \qquad T(t + s) = T(t) T(s) , \qquad \forall t, s \geq 0 , \]
and such that the function $t \mapsto T(t)$ is strongly continuous.
 Each $C_0$-semigroup $\mathcal{T}$ possesses an infinitesimal generator $A$, which is a closed operator $A : D(A) \subseteq X \to X$ on the dense subspace $D(A)$ of $X$ for which
	\[ Ax = \lim_{t \to 0^+} \frac{T(t) x - x}{t} , \qquad x \in D(A) . \]
Throughout the remainder of this article, unless otherwise specified, when considering a $C_0$-semigroup $\mathcal{T} = (T(t))_{t\geq0}$, $X$ will denote the Banach space it acts on, and $A$ its generator with domain $D(A)$.

The resolvent set $\rho(A)$ of a $C_0$-semigroup $\mathcal{T}$ is the set of all $\lambda \in \C$ such that $\lambda - A : D(A) \to X$ is an isomorphism. Then $\rho(A)$ is an open set, and its resolvent at the point $\lambda \in \rho(A)$ is the operator
	\[ R(\lambda, A) = (\lambda - A)^{-1}  . \]
Since $A$ is closed, $R(\lambda,A)\in L(X)$. It defines a holomorphic function $\rho(A) \to L(X)$. Its spectrum is $\sigma(A) = \C \setminus \rho(A)$. We also consider the residual spectrum $\sigma_R(A) = \{ \lambda \in \sigma(A) \mid \rg(\lambda - A) \text{ is not dense in } X \}$.

A $C_0$-semigroup $\mathcal{T}$ is called \emph{bounded} if $\sup_{t \geq 0} \|T(t)\|_{L(X)} < \infty$.
In this case, the orbit function $t \mapsto T(t) x$ is bounded and uniformly continuous for any $x \in X$.
By the integral representation formula for the resolvent \cite{E-N-ShortCourseSemigroups}, if a $C_0$-semigroup $\mathcal{T}$ is bounded, then $\C_+ \subseteq \rho(A)$ and we get the global bound
	\begin{equation}\label{eq:resolvent bound} 
		\|R(\lambda, A)\|_{L(X)} =O\left(\frac{1}{\alpha}\right) , \qquad \lambda = \alpha + i \xi \in \C_+ . \end{equation}
Therefore, the resolvent has distributional boundary values on the whole imaginary axis, and its boundary distribution is an $L(X)$-valued tempered distribution. According to our convention, we denote this boundary distribution as $R(i\xi,A)\in \mathcal{S}'(\R; L(X))$. 

For any $x \in X$, the \emph{local unitary spectrum $\sigma_u(A, x)$} \cite{B-V-R-LocalSpecInStab} is defined as the set of those points $i \xi_0 \in i \R$ where $R(\lambda, A) x$ does not allow a holomorphic extension to any neighborhood of $i \xi_0$. 
Clearly $\sigma_u(A, x) \subseteq \sigma(A) \cap i \R$, while the uniform boundedness principle yields
	\[ \sigma(A) \cap i \R = \bigcup_{x \in X} \sigma_u(A, x) . \]
Other types of boundary behavior, less restrictive than analytic continuation,  may occur on the imaginary axis for the resolvent. This motivates the following definition. Note that, for a bounded semigroup, $R(i \xi,A)x=\mathcal{F}\{T(t)x; \xi\}$ is an $X$-valued tempered distribution for each $x\in X$.

\begin{definition}
\label{def:LocalPseudofuncSpect}
Let $\mathcal{T}$ be a bounded $C_0$-semigroup.

$(i)$ For $x \in X$, we denote the \emph{local pseudofunction spectrum of $A$ at $x$} as
	\[ \sigma_{\PF}(A, x) = i \singsupp_{\PF} R(i \, \cdot, A) x . \]
	
$(ii)$ For any $M \subseteq X$, we write $\sigma_{\PF}(A, M)$ for the complement of the largest open subset $i I \subseteq i \R$ such that $R(i \xi, A) x \in \PFloc(I; X)$ for every $x \in M$.
For the special case $M=X$, we simply employ the notation
\[\sigma_{\PF}(A)=\sigma_{\PF}(A,X).\]
\end{definition}

\begin{remark}
\label{r:InclusionsSpectra}
\

$(i)$ In the terminology from Definition \ref{def:LocalPseudofuncBoundBeh}, Definition \ref{def:LocalPseudofuncSpect}(i) can be rephrased as follows: $\sigma_{\PF}(A, x)$ is the smallest closed subset of $i\R$ such that $R(\lambda, A)$ has local pseudofunction boundary behavior on the boundary open subset $i\R \setminus \sigma_{\PF}(A, x)$.

$(ii)$ Since $R(\lambda, A)$ has analytic extension to a neighborhood of $\rho(A) \cap i \R$, it has, in particular, local pseudofunction boundary behavior there. Therefore, $\sigma_{\PF}(A) \subseteq \sigma(A) \cap i \R$.
In fact, for each $x \in X$, we have $\sigma_{\PF}(A, x) \subseteq \sigma_{u}(A, x)$, but the inclusion can be strict, see Examples \ref{ex:PFEmptyButNotContinuous} and \ref{ex:PseudoZeroUnitaryWholeLine} below.

$(iii)$ For $M' \subseteq M \subseteq X$ we of course have $\sigma_{\PF}(A,M') \subseteq \sigma_{\PF}(A,M)$.
Therefore, for any $x \in M \subseteq X$,
	\begin{equation}
		\label{eq:InclusionChainSpectra} 
		\sigma_{\PF}(A,x) \subseteq \sigma_{\PF}(A,M) \subseteq \sigma_{\PF}(A) \subseteq \sigma(A) \cap i \R . 
	\end{equation}
	
$(iv)$ By \eqref{eq:InclusionChainSpectra}, one has $\bigcup_{x \in M} \sigma_{\PF}(A,x) \subseteq \sigma_{\PF}(A,M)$ for any arbitrary subset $M \subseteq X$. 
However, equality does not need to hold, as can be seen from Example \ref{ex:UnionNotPseudoSpectrum}, unless $M $ is a Baire subspace, see Proposition \ref{l:LocalPseudofunctionSpectrumIsUnion}.
On the other hand, it is clear that
	\[ \sigma_{\PF}(A, M) = \overline{\bigcup_{x\in M}\sigma_{\PF}(A,x)} . \] 

$(v)$ The operator $A$ must not necessarily be the infinitesimal generator of a bounded $C_{0}$-semigroup in order to be able to define the local pseudofunction spectra as in Definition \ref{def:LocalPseudofuncSpect}.  In general, our definitions of $\sigma_{\PF}(A, x)$ and $\sigma_{\PF}(A, M)$ above still make sense for any closed operator $A$ on a Banach space $X$ (with non-necessarily dense domain) as long as $\C_{+}\subseteq \rho(A)$ and its resolvent satisfies local bounds
\[
\|R(\lambda,A)\|_{L(X)}=O\left( \frac{1}{\alpha^{N}}\right), \qquad \lambda=\alpha + i\xi \in (0,\alpha_0] + i I,
\]
for each bounded interval $I\subset \R$ and some $N=N_{I}$, ensuring that $R(i\xi,A)\in\mathcal{D}'(\R;L(X))$ exits (see \eqref{BoundaryValuesBound}). For instance, this applies to infinitesimal generators of polynomially bounded semigroups $\mathcal{T}$.
However, for our analysis in the remainder of this article, the assumption of boundedness for the semigroup and that $A$ arises as its infinitesimal generator will often be needed, see Example \ref{ex:BoundCondNess1}. For this reason, our standing assumption will be to work with bounded $C_0$-semigroups.

\end{remark}

Let us show three basic but important properties of local pseudofunction spectra. They play a crucial role in our stability study of bounded $C_0$-semigroups carried out in the next sections.

\begin{proposition}\label{PFClosure}
Let $\mathcal{T}$ be a bounded $C_0$-semigroup and let $M \subseteq X$. Then 
	\begin{equation}
		\label{eq:PFClosure}
		\sigma_{\PF}(A, \overline{M}) = \sigma_{\PF}(A, M) .
	\end{equation}  
\end{proposition}

\begin{proof}
By Remark \ref{r:InclusionsSpectra}{$(iii)$}, it suffices to show that $\sigma_{\PF}(A, \overline{M}) \subseteq \sigma_{\PF}(A, M)$.
Fix $i I = i \R \setminus \sigma_{\PF}(A, M)$ and take any $\varphi \in \D(I)$.
Then    
    \[  \lim_{|t| \to \infty} \ev{R(i \xi, A) y}{e^{it\xi} \varphi(\xi)} = 0 , \qquad \forall y \in M  . \] 
For any $x \in \overline{M}$, we may take a sequence $(x_n)_{n \in \N}$ in $M$ converging to $x$.
By dominated convergence, we have	
    \begin{align*}
    	\lim_{n \to \infty} \langle R(i\xi, A) x_n, e^{it\xi} \varphi(\xi) \rangle 
	&= \lim_{n \to \infty} \int^\infty_0 \widehat{\varphi}(s - t)T(s) x_n  \, \mathrm{d}s \\
	&= \int_0^\infty \widehat{\varphi}(s - t)T(s)x \, \mathrm{d}s  
	= \ev{R(i\xi, A) x}{e^{it\xi} \varphi(\xi)} .
    \end{align*}
Moreover, the convergence is uniform in $t$ because
   \begin{equation}
   \label{eq:uniform switiching auxiliary}
     	\sup_{t \in \R} \Big\Vert\int^\infty_0 \widehat{\varphi}(s-t) T(s)(x - x_n)\, \mathrm{d}s \Big\Vert_X \leq \left(\sup_{s \geq 0} \|T(s)\|_{L(X)}\right) \Vert x-x_n\Vert_{X} \Vert\widehat{\varphi}\Vert_{L^1} .
    \end{equation}
Therefore,
    \begin{align*}
            \lim_{|t| \to \infty} \langle R(i\xi, A)x, e^{it\xi} \varphi(\xi) \rangle
            &=\lim_{|t| \to \infty} \lim_{n \to \infty} \langle R(i\xi, A)x_n, e^{it\xi} \varphi(\xi)\rangle \\
            &=\lim_{n \to \infty} \lim_{|t| \to \infty} \langle R(i\xi, A)x_n, e^{it\xi} \varphi(\xi)\rangle 
            = 0 .
    \end{align*}
Hence, $R(i \xi, A) x \in \PFloc(I; X)$ for all $x \in \overline{M}$ and \eqref{eq:PFClosure} holds true.
\end{proof}

\begin{proposition}
	\label{l:LocalPseudofunctionSpectrumIsUnion}
	For a bounded $C_0$-semigroup $\mathcal{T}$ and a Baire linear subspace $M \subseteq X$, 	\begin{equation}\label{eq:UnionIsPseudoSpectrum}
	\sigma_{\PF}(A,M) = \bigcup_{x\in M}\sigma_{\PF}(A,x) .
	\end{equation}
	\end{proposition}
	\begin{proof}
	The inclusion $\bigcup_{x\in M}\sigma_{\PF}(A,x)\subseteq \sigma_{\PF}(A,M)$ follows from \eqref{eq:InclusionChainSpectra}. 
	Let now $i\xi_0 \notin \sigma_{\PF}(A,x)$ for any $x \in M$. 
	Define, for every $j \in \N$, the interval $I_j = \left(\xi_{0} - \frac{1}{j}, \xi_{0} + \frac{1}{j}\right)$ and consider the subspaces
		\[ V_j = \{x \in X \mid R(i\,\cdot,A)x \in \PFloc(I_j; X) \} , \qquad j \in \N. \]
	By our assumption on $i \xi_0$, we obtain $\bigcup_{j = 1}^\infty (V_j\cap M) = M$.
	Also, each subspace $V_j$ is closed.
	Indeed, let $V_j \ni x_n \to x$. 
	Then
		\[ \lim_{\vert t\vert\to \infty} \ev{R(i\xi,A)x_n}{e^{it\xi}\varphi(\xi)} = 0 , \qquad \forall \varphi \in \D(I_j) . \] 
Hence, in view of \eqref{eq:uniform switiching auxiliary}, we find, for any $\varphi \in \D(I_j)$
		\begin{align*}
			\lim_{\vert t\vert\to \infty}\ev{R(i\xi,A)x}{e^{it\xi}\varphi(\xi)}
			&=\lim_{\vert  t\vert\to\infty}\lim_{n\to\infty}\ev{R(i\xi,A)x_n}{e^{it\xi}\varphi(\xi)} \\
			&=\lim_{n\to\infty}\lim_{\vert  t\vert\to \infty}\ev{R(i\xi,A)x_n}{e^{it\xi}\varphi(\xi)}=0 .
		\end{align*} 
	Thus $x \in V_j$. 
	
	Since $M$ is a Baire space and $M\cap V_j$ are closed subspaces of it, we must have $M= V_j\cap M$ for some $j \in \N_{+}$.
	Hence, $i\xi_0 \notin \sigma_{\PF}(A, M)$.
	As $i \xi_0$ was chosen arbitrarily, we deduce the sought inclusion.
	\end{proof}

Note that Proposition \ref{l:LocalPseudofunctionSpectrumIsUnion} holds in particular for closed subspaces of $X$. Especially,
	\begin{equation}
		\label{eq:UnionIsFullPseudoSpectrum}
		\sigma_{\PF}(A) = \bigcup_{x \in X} \sigma_{\PF}(A, x) .
	\end{equation}

We now show that operators commuting with $R(\lambda, A)$ shrink the size of the local pseudofunction spectrum when applied to a vector. In the next proposition, we endow $D(A)$ with a Banach space topology using the graph norm $\|x\|_{D(A)} = \|x\|_X + \|Ax\|_X$.

\begin{proposition}
	\label{l:CommutingOperator}
	Let $\mathcal{T}$ be a bounded $C_0$-semigroup.
	Suppose $x \in X$ is such that there exist $y \in X$ and $L \in L(D(A), X)$ for which
		\begin{equation}
			\label{eq:CommuteWithResolvent}
			R(\lambda, A) x = L R(\lambda, A) y , \qquad \lambda \in \C_+ . 
		\end{equation}
	Then 
		\[ \sigma_{\PF}(A, x) \subseteq \sigma_{\PF}(A, y)  . \]
\end{proposition}

\begin{proof}
For any $z \in X$ we have
	\begin{equation}
		\label{eq:SwitchingAAndResolvent}
		A R(\alpha + i\xi, A) z = -z + (\alpha + i\xi) R(\alpha + i\xi, A) z , \qquad \alpha > 0, ~ \xi \in \R . 
	\end{equation}
Therefore, for fixed $\psi \in \D(\R)$ and $\alpha > 0$, the mapping
	\[ \R \ni \xi \mapsto \psi(\xi) R(\alpha + i\xi,{A}) z \in D(A) \]
is a continuous $D(A)$-valued map with compact support, so that the $D(A)$-valued integral $\int_{-\infty}^\infty \psi(\xi) [R(\alpha + i\xi, A) z] \mathrm{d}\xi$ exists.
Now
	\begin{equation}
		\label{eq:ResolventIntegralConvergence} 
		\int_{-\infty}^\infty \psi(\xi) [R(\alpha + i\xi, A) z] \mathrm{d}\xi \to \langle R(i\xi, A) z, \psi(\xi) \rangle , \qquad \text{as } \alpha \to 0^+ \text{ in } X .
	\end{equation}
Also, by \eqref{eq:SwitchingAAndResolvent},
	\begin{align*} 
		A \int_{-\infty}^\infty \psi(\xi) [R(\alpha + i\xi, A) z] \, \mathrm{d}\xi 
		&= \int_{-\infty}^\infty \psi(\xi) [A R(\alpha + i\xi, A) z] \, \mathrm{d}\xi \\
		&= - \widehat{\psi}(0) z + \int_{-\infty}^\infty (\alpha + i\xi) \psi(\xi) [R(\alpha + i\xi, A) z] \, \mathrm{d}\xi \\
		&\to - \widehat{\psi}(0) z + \langle R(i\xi, A) z, i \xi \psi(\xi) \rangle ,
	\end{align*}
as $\alpha \to 0^+$ in $X$.
Since $A$ is a closed operator, it must follow that $\langle R(i\xi, A) z, \psi(\xi) \rangle \in D(A)$, that the convergence in \eqref{eq:ResolventIntegralConvergence} holds in $D(A)$, and that
	\begin{equation} 
		\label{eq:AOfBoundary}
		A \langle R(i\xi, A) z, \psi(\xi) \rangle = - \widehat{\psi}(0) z + \langle R(i\xi, A)z, i \xi \psi(\xi) \rangle .
	\end{equation}
	
Let $I \subseteq \R$ be open such that $R(i \xi, A) y \in \PFloc(I; X)$ and take $\varphi \in \D(I)$ arbitrarily.
Using \eqref{eq:CommuteWithResolvent} and our observation that the convergence in \eqref{eq:ResolventIntegralConvergence} holds in $D(A)$, it follows that, in $D(A)$,
	\begin{align*} 
		\langle R(i\xi, A) x, e^{i t \xi} \varphi(\xi)\rangle 
		&= \lim_{\alpha \to 0^+} \int_{-\infty}^\infty e^{i t \xi}  \varphi(\xi) [R(\alpha+i\xi, A) x]  \, \mathrm{d}\xi \\
		&= \lim_{\alpha \to 0^+} \int_{-\infty}^\infty L\left(e^{i t \xi} \varphi(\xi) [R(\alpha + i\xi, A) y] \right) \, \mathrm{d}\xi \\
		&= \lim_{\alpha \to 0^+} L \int_{-\infty}^\infty e^{i t \xi} \varphi(\xi) [R(\alpha + i\xi, A) y] \, \mathrm{d}\xi \\
		&= L \langle R(i\xi, A)y, e^{i t \xi}  \varphi(\xi)\rangle  .
	\end{align*}
Now $\langle R(i\xi, A)y, e^{i t \xi}  \varphi(\xi)\rangle \to 0$ as $\lvert t \rvert \to \infty$ in $X$, while also, by \eqref{eq:AOfBoundary},
	\[ A \langle R(i\xi, A)y, e^{i t \xi}  \varphi(\xi)\rangle = - \widehat{\varphi}(-t) y + \langle R(\alpha + i\xi, A) y, e^{i t \xi} i \xi \varphi(\xi) \rangle \to 0 , \]
as $\lvert t \rvert \ \to \infty$ in $X$.
Therefore,
	\[ \langle R(i\xi, A) y, e^{i t \xi}  \varphi(\xi)\rangle \to 0 , \qquad \text{as } \lvert t\rvert \ \to \infty \text{ in } D(A) . \]
It follows that
	\[ \lim_{\lvert t\rvert \ \to \infty} \langle R(i\xi, A)x, e^{i t \xi} \varphi(\xi)\rangle  = \lim_{\lvert t\rvert \ \to \infty} L \langle R(i\xi, A)y, e^{i t \xi}  \varphi(\xi)\rangle = 0 . \]
As $\varphi \in \D(I)$ was chosen arbitrarily, we conclude $R(i \xi, A) x \in \PFloc(I; X)$.
\end{proof}

\begin{corollary}
	\label{c:SpectrumOf(A-s)}
	Let $\mathcal{T}$ be a bounded $C_0$-semigroup.
	Then $\sigma_{\PF}(A, (\lambda_0 - A) x) \subseteq \sigma_{\PF}(A, x)$ for any $\lambda_0 \in \C$ and $x \in D(A)$.
\end{corollary}

We conclude this section with three illustrative examples. We start by showing that in general $\sigma_{\PF}(A, x) \subsetneq \sigma_u(A, x)$. In fact, even if $\sigma_{\PF}(A, x) = \varnothing$, this does not necessarily guarantee that $\sigma_u(A, x)$ should be empty. We give two examples of the latter situation. In our first example, $R(i\xi, A)x$ is an $X$-valued pseudofunction that is discontinuous at $\xi=0$. 
We need the following lemma, which provides a family of instances of local pseudofunctions that are not $L^{1}_{\loc}$ at the origin.

\begin{lemma}
	\label{l:SuffCondNonContFourier}
	Let $f\in L^1_{\loc}(\R_+)$ be an eventually positive non-increasing function that satisfies $f(t)=o(1)$ as $t \to \infty$. 
	Suppose
		\begin{equation}
			\label{eq:Intf(t)/t=infty}
			\int^\infty_1\frac{f(t)}{t} \mathrm{d}t=\infty.
		\end{equation}
	Then $\widehat{f}\in \PFloc(\R)$ is continuous on $\R \setminus \{0\}$, but $\widehat{f} \notin L^{1}_{\loc}(\R)$.
\end{lemma}
	
\begin{proof}
		Without loss of generality, we may assume that $f \in C^1[0,2]$, that $f$ is non-increasing on $[1,\infty)$, and $f \equiv 0$ around the origin.
		First, we show the continuity of $\widehat{f}$ on $\R \setminus \{0\}$. 
		Note that its distributional derivative $f'(t)=\mathrm{d}f(t)$ is a measure of finite variation (on the whole of $\R$).
Thus $\widehat{f'} \in C(\R)$, and since $i \xi \widehat{f}(\xi) = \widehat{f'}(\xi)$, we see that $\widehat{f}$ is continuous away from the origin.
		
		Suppose now $\widehat{f}\in L^1(-1,1)$.
		Then $\widehat{f}$ has a continuous primitive $g$.
		By \cite[Theorem 10]{V-E-DistPointValConvergFourierSeriesInt}, there exists a $k \in \N$ such that we have the Ces\`aro limit
			\[ 
				2\pi g(0)
				= \lim_{r\to\infty}\int_{-r}^{r}\widehat{g}(\xi)\Big(1-\frac{\vert \xi \vert}{r}\Big)^k \, \mathrm{d}\xi .
			\]
		Therefore, we also find the Abel limit
			\[ 2\pi g(0) = \lim_{\varepsilon \to 0+} \int^\infty_{-\infty} e^{- \varepsilon \vert \xi \vert} \widehat{g}(\xi) \, \mathrm{d}\xi . \]
		Now 
		\[
		\widehat{g}(t) = 2\pi \left(\frac{f(-t)}{it} + g(0) \delta(t)\right),
		\]
		with $\delta$ the Dirac delta distribution.
	The monotone convergence theorem yields
			\[\int^\infty_1 \frac{f(t)}{t} \, \mathrm{d}t<\infty,\]
		which contradicts \eqref{eq:Intf(t)/t=infty}.
		Consequently, $\widehat{f}$ cannot be integrable on $(-1, 1)$.
	\end{proof}
	
\begin{example}
\label{ex:PFEmptyButNotContinuous}
There are bounded $C_0$-semigroups $\mathcal{T}$ with $x \in X$ such that $\sigma_{\PF}(A, x) = \varnothing$, but for which $R(i \, \cdot, A) x$ is not continuous at $0$. 
In particular, $0 \in \sigma_u(A,x)$.
\end{example}

\begin{proof}
	Take any $f \in C_0(\R_{+})$ satisfying the conditions of Lemma \ref{l:SuffCondNonContFourier}.
	A concrete example of such a function is given by $f(t)=1/\log(t+2)$.
	We consider the left-translation semigroup $\mathcal{T}$ on the Banach space $C_0(\R_{+})$. 
	Then $\mathcal{T}$ is clearly a bounded $C_0$-semigroup for which $\sigma(A) \cap i \R = i \R$. Furthermore, this semigroup is strongly stable. In particular, $\lim_{t\to\infty}T(t)f=0$, so that $R(i \xi, A) f=\mathcal{F}\{ T(t)f; \xi\} \in \PF(\R; C_{0}(\R_{+}))$ and thus $\sigma_{\PF}(A, f) = \varnothing$.
 Let us explicitly compute the $C_{0}(\R_{+})$-valued distribution $R(i \xi, A) f$.  
 For $\lambda \in \C_+$ and $s\in\R_{+}$,
		\[
			(R(\lambda,A)f)(s)
			= \int^\infty_0 e^{-\lambda t}f(t + s) \, \mathrm{d}t
			= e^{\lambda s}\int^\infty_s e^{-\lambda t} f(t) \, \mathrm{d}t
			= e^{\lambda s} \left( \L\{f; \lambda\} - \int_0^s e^{-\lambda t}f(t) \, \mathrm{d}t \right) .
		\]
	Therefore, its boundary value distribution is given by
		\[ (R(i \xi, A) f)(s) = e^{i \xi s} \left( \widehat{f}(\xi) - \int_0^s e^{- i \xi t} f(t) \mathrm{d}t \right) . \]
Since $\widehat{f}\in C(\R\setminus\{0\})$ in view of Lemma \ref{l:SuffCondNonContFourier}, we infer that $R(i \xi, A)f\in C(\R\setminus\{0\}, C_{0}(\R_{+}))$. 
On the other hand, again by Lemma \ref{l:SuffCondNonContFourier}, for a fixed $s$, the scalar-valued distribution $(R(i \xi, A) f)(s)$ is not continuous at $\xi=0$.
	As a result, neither can the $C_0(\R_{+})$-valued distribution $R(i \xi, A) f$ be continuous at the origin.
\end{proof}

Next, we see that even the most extreme case can occur: $\sigma_{\PF}(A, x) = \varnothing$ but $\sigma_u(A, x) = i \R$. 
The following family of examples has already been considered in connection with the study of stable semigroups, see \cite{B-V-R-LocalSpecInStab,C-T-StabilityGeometry}. Batty, Van Neerven, and R\"abiger used them \cite[Example 4.1]{B-V-R-LocalSpecInStab} to exhibit the existence of stable $C_0$-semigroups for which $\sigma_u(A, x) = i \R$ for any $x \neq 0$, refuting so a conjecture of Huang \cite{Hu93}; while Chill and Tomilov even showed \cite[Proposition 4.3]{C-T-StabilityGeometry} that the choice of the weight function $w(s)=(\log (e+s))^{-1}$ delivers a semigroup with 
		$ \bigcap_{\xi\in\R} \rg(i \xi - A)=\{0\}$, so that \eqref{eq: local range} fails to characterize strong stability.

\begin{example}
\label{ex:PseudoZeroUnitaryWholeLine}
Let $X=L^1(\R_+, w(s) \mathrm{d}s)$, where $w:\R_+\to \R_+$ satisfies
\begin{enumerate}
	\item $w$ is non-increasing,
	\item $\lim_{s\to\infty}w(s)=0$,
	\item for each $a>0$, there exists a constant $c>0$ such that $w(s)\geq ce^{-as}$, for all $s\geq0$.
	\end{enumerate}
Let $\mathcal{T}$ be the right-translation semigroup on $X$ defined as
\begin{equation}
	\label{eq:RightTransSemigroup}
	(T(t)f)(s)
	=
		\begin{cases}
			f(s-t) , &s \geq t,\\
			0 ,  & s < t,
		\end{cases}
	\qquad
	f \in L^1(\R_+, w(s) \mathrm{d}s) .
\end{equation}
By (1), $\mathcal{T}$ is a bounded $C_0$-semigroup on $X$. 
Then $\sigma_u(A, f) = i \R$ for all $f \in X \setminus \{0\}$, see \cite[Example 4.1]{B-V-R-LocalSpecInStab}.
However, we now verify that $\sigma_{\PF}(A, f) = \varnothing$ for every $f \in X$. 
In fact, if $f \in X$ has compact support, the support of $T(t)f$ is that of $f$ translated by $t$, so (2) yields that $T(t)f = o(1)$ as $t \to \infty$.
Since $\mathcal{T}$ is bounded, a density argument shows that $T(t)f=o(1)$ as $t\to\infty$ for each $f\in X$. Therefore, as $R(i\xi, A)f$ is the Fourier transform of $T(t)f$, it is a global $X$-valued pseudofunction and we thus obtain $\sigma_{\PF}(A, f) = \varnothing$ for every $f \in X$. Naturally, we also obtain here that $\sigma_{\PF}(A) = \varnothing$ but $\sigma(A) \cap i \R = i \R$. 
\end{example}

We now provide examples for which the equality \eqref{eq:UnionIsPseudoSpectrum} fails. This family of examples already appeared in \cite[Example 2.5(a)]{A-B-TaubThStabSemigr} for different illustrative purposes under stronger constraints on the Borel measure (see also Remark \ref{ExampleConditionMainResult} below).

\begin{example}
\label{ex:UnionNotPseudoSpectrum}
Let $ w$ be a non-empty closed subset of $\R$ and let $\mu$ be a non-trivial non-negative Borel measure on it. 
We consider the space $X = L^2( w, \mu)$ and the multiplication semigroup $(T(t) f)(u) = e^{itu} f(u)$.
It is an isometric $C_0$-semigroup on $X$ whose generator is given by $f(u) \mapsto i u f(u)$.

We claim that
	\begin{equation}
		\label{eq:PFSpec=EssSupp}
		\sigma_{\PF}(A, f) = i \, \text{ess}\supp(f) , \qquad f \in L^2( w, \mu) ,
	\end{equation}
	where the essential support is taken with respect to $\mu$. Let $f \in L^2( w, \mu)$ be fixed.
For any $\varphi \in \D(\R)$,
	\begin{align*}
        		\| \ev{R(i \xi,A) f}{e^{it\xi} \varphi(\xi)} \|^2_{X} 
		&=\int_{ w} \left| \int^\infty_0\widehat{\varphi}(s-t) (e^{isu} f(u))  \,  \mathrm{d}s \right|^2 \, \mathrm{d}\mu(u) \\
		&= \int_ w |f(u)|^2 \left| \int_{-t}^\infty  e^{isu} \widehat{\varphi}(s) \, \mathrm{d}s \right|^2 \, \mathrm{d}\mu(u) .
    \end{align*}
By the Lebesgue dominated convergence theorem, 
	\[ 
		\lim_{t \to -\infty}  \| \ev{R(i \xi,A) f}{e^{it\xi} \varphi(\xi)} \|^2_{X}  = 0 
	\]
and
	\[
		\lim_{t \to +\infty} \| \ev{R(i \xi,A) f}{e^{it\xi} \varphi(\xi)} \|^2_{X}  = 2 \pi \int_{ w } |\varphi(u)|^2|f(u)|^2  \, \mathrm{d}\mu(u).
	\]
Directly, one sees $\sigma_{\PF}(A,f)\subseteq i \, \text{ess}\supp(f)$.
Now take any $\xi_0 \in \text{ess}\supp(f)$. Let $I$ be an open interval containing the point $\xi_0$, then $\mu(\{u\in I\cap  w\mid |f(u)|>0 \})>0$. There are then $\varepsilon>0$ and a compact subset $K\subset I\cap w$ such that $\mu(K)>0$ and $|f(u)|\geq \varepsilon$ for all $u\in K$.
Then, for any $\varphi \in \D(I)$ satisfying $\varphi \equiv 1$ on $K$, 
	 \[ \int_{ w }|\varphi(u)|^2 |f(u)|^2  \, \mathrm{d}\mu(u) \geq\int_{K} |f(u)|^2 \, \mathrm{d}\mu(u) \geq \mu(K) \varepsilon^{2} > 0. \]
Consequently, $R(i \xi, A) f \notin \PFloc(I; L^2( w, \mu))$.
As $I$ was arbitrary, we see that $i \, \text{ess}\supp(f)\subseteq\sigma_{\PF}(A,f)$.
We have thus established \eqref{eq:PFSpec=EssSupp}.

Using \eqref{eq:PFSpec=EssSupp}, we may now easily show that \eqref{eq:UnionIsPseudoSpectrum} does not hold in general.
Let $ w = [0, 1]$ and $\mu$ be the Lebesgue measure on it.
Let $\chi_n$ be the characteristic function of $[\frac{1}{n}, 1]$ for $n \in \N$.
Set $M = \{\chi_n \mid n \in \N\}$. 
By \eqref{eq:PFSpec=EssSupp}, it holds that 
  	\[ \bigcup_{n = 1}^\infty \sigma_{\PF}(A,\chi_n) = i \, \bigcup_{n = 1}^\infty \text{ess}\supp(\chi_n) = i \, \bigcup_{n = 1}^\infty \left[\frac{1}{n}, 1\right] = i \, (0, 1] .\]
As $\sigma_{\PF}(A,M)$ is necessarily closed (in fact, it is equal to $[0, 1]$), the two sets do not coincide.
\end{example}

\section{The stability of orbits}
\label{sec:StabOrbits}

Let $\mathcal{T}$ be a $C_0$-semigroup and $x \in X$. We call $\{ T(t) x \mid t \geq 0 \}$ the \emph{orbit of $x$ under $\mathcal{T}$}, and often simply write $T(t)x$ for it.
An orbit $T(t) x$ is said to be \emph{stable} if 
	\[ \lim_{t \to \infty} \| T(t) x \|_X = 0 . \]

We shall now provide several characterizations of the stability of an orbit for bounded $C_0$-semigroups in terms of the local pseudofunction spectrum at a vector (and local properties of the resolvent). The following theorem is the main result of this section.
We employ the notation ``$\wlim{}$'' for a weak limit in the Banach space.

\begin{theorem}
\label{t:StableOrbitCountable}
Let $\mathcal{T}$ be a bounded $C_0$-semigroup and let $x\in X$. The following statements are equivalent:
	\begin{itemize}
		\item[$(a)$] The orbit $T(t)x$ is stable;
	\item[$(b)$] $\sigma_{\PF}(A, x) = \varnothing$;
		\item[$(c)$] $\sigma_{\PF}(A, x)$ is countable and $x \in \overline{\rg(i\xi - A)}$ for each $i \xi \in \sigma_{\PF}(A, x)$;
		\item[$(d)$] $\sigma_{\PF}(A, x)$ is countable and $\lim_{\alpha\to0^+} \alpha R(\alpha + i\xi, A) x=0$ for each $i \xi \in \sigma_{\PF}(A, x)$;
		\item[$(e)$] $\sigma_{\PF}(A, x)$ is countable and $\wlim{\alpha\to0^+} \alpha R(\alpha + i\xi, A) x=0$ for each $i \xi \in \sigma_{\PF}(A, x)$.
	\end{itemize}
\end{theorem}

Theorem \ref{t:StableOrbitCountable} considerably improves upon the local version of the ABLV theorem \cite[Theorem 3.4]{B-V-R-LocalSpecInStab}.  Example \ref{ex:PseudoZeroUnitaryWholeLine} shows that there are bounded $C_0$-semigroups for which one cannot detect their orbit stability through the local unitary spectrum; moreover, \cite[Theorem 3.4]{B-V-R-LocalSpecInStab} is not applicable to such a strongly stable semigroup. In contrast, Theorem \ref{t:StableOrbitCountable} provides the right framework for characterizing stability of orbits through the local pseudofunction spectrum.

We need two lemmas in preparation for the proof of Theorem \ref{t:StableOrbitCountable}. Our first auxiliary lemma is the following version of the mean ergodic theorem.

\begin{lemma}
\label{l:Ergodic}
Let $\mathcal{T}$ be a bounded $C_0$-semigroup.
For any $\xi_0 \in \R$ and $x \in X$, the following are equivalent:
	\begin{itemize}
		\item[$(a)$] $x \in \overline{\rg{(i \xi_0 - A)}}$;
		\item[$(b)$] $\lim_{\alpha\to0^+} \alpha R(\alpha+ i \xi_0, A)x = 0$;
		\item[$(c)$] $\wlim{\alpha \to 0^+} \alpha R(\alpha + i \xi_0, A) x = 0$.
	\end{itemize}
If the statements hold true, we have
	\begin{equation}
		\label{ApproximatedNet}
		\rg{(i\xi_0 - A)} \ni x_\alpha :=(i\xi_0 - A)R(\alpha+i\xi_0, A)x \to x , \qquad \text{as } \alpha \to 0^+ .
	\end{equation}
\end{lemma}

\begin{proof}
$(a) \Longrightarrow (b)$ is shown in, for example, \cite[Proposition 4.3.21, p.~261]{A-B-H-N-VVLaplaceTransCauchyProb}, while $(b) \Longrightarrow (c)$ is trivial.
Suppose now $\wlim{\alpha \to 0^+} \alpha R(\alpha + i \xi_0, A) x = 0$ and let $x_\alpha$ be as in \eqref{ApproximatedNet} for $\alpha > 0$.
For any $x' \in X^\prime$,
	\[ \lim_{\alpha \to 0^+} \ev{x'}{x - x_\alpha} =  \lim_{\alpha \to 0^+} \ev{x'}{\alpha R(\alpha + i\xi_0, A) x} = 0 . \]
Therefore, $x$ lies in the weak closure of $\rg{(i \xi_0 - A)}$.
The latter is a subspace of $X$, so it coincides with $\overline{\rg{(i \xi_0 - A)}}$ in view of the Hahn-Banach theorem.
\end{proof}

Next, we observe that when any of the equivalent statements from Lemma \ref{l:Ergodic} holds,  the point $i\xi_0$ cannot be an isolated point of the local pseudofunction spectrum of $A$ at $x$.

\begin{lemma}
\label{LocalVersionIsolatedPoints}
Let $\mathcal{T}$ be a bounded $C_0$-semigroup.
Let $x \in X$ and $\xi_0 \in \R$ be such that $x\in \overline{\rg(i \xi_0 - A)}$.
Then, $i \xi_0$ is not an isolated point of $\sigma_{\PF}(A, x)$.
\end{lemma}

\begin{proof}
We assume first that $x \in \rg(i \xi_{0} - A)$. Suppose that there is a non-empty interval $I = (\xi_0 - \varepsilon, \xi_0 + \varepsilon)$ such that  
\begin{equation}
\label{eq:specisocondpoint}
[ i (I \setminus \{\xi_0\})] \cap \sigma_{\PF}(A, x) = \varnothing . 
\end{equation}
We must show that $i\xi_0 \notin \sigma_{\PF}(A, x)$. Our assumption \eqref{eq:specisocondpoint} is the same as $R(i \xi, A)x=\mathcal{F}\{T(t)x;\xi\}\in \PFloc(I \setminus\{\xi_0\}; X)$. We verify that Theorem \ref{t:SuffCondVanishSingSupp} applies to $\tau(t)=T(t)x$ with $E=\{\xi_0\}$. In fact, writing $x = (i \xi_0 - A) y$ with $y\in D(A)$, we have
	\begin{align*} 
		\int_0^t e^{- i \xi_0 s}T(s) x \, \mathrm{d}s 
		&= \int_0^t e^{- i \xi_0 s}T(s) (i \xi_0 - A) y   \, \mathrm{d}s 
		= - \int_0^t \frac{\mathrm{d}}{\mathrm{d}s} [ e^{- i \xi_0 s}T(s) y] \, \mathrm{d}s 
		\\
		&= y - e^{- i \xi_0 t}T(t) y=O(1). 
	\end{align*}
Theorem 	\ref{t:SuffCondVanishSingSupp} thus yields $i I \cap \sigma_{\PF}(A, x) = \varnothing $, completing the proof of the lemma in this case.

Suppose now that $x \notin \rg(i \xi_0 - A)$ and let $x_\alpha$ be the approximation net as in \eqref{ApproximatedNet}.
Let $\varepsilon>0$ be such that $[i(I \setminus \{\xi_0\})] \cap \sigma_{\PF}(A, x) = \varnothing$ where $I = (\xi_0 - \varepsilon, \xi_0 + \varepsilon)$.
By Proposition \ref{l:CommutingOperator}, with $L = (i \xi_0 - A) R(\alpha + i \xi_0, A)$, we must also have $[i(I \setminus \{\xi_0\})] \cap \sigma_{\PF}(A, x_\alpha) = \varnothing$ for $\alpha > 0$.
Consequently, the first case already analyzed above yields $iI \cap \sigma_{\PF}(A, x_\alpha) = \varnothing$ for $\alpha > 0$.
Set $M = \{ x_\alpha \mid \alpha > 0 \}$, then $iI \cap \sigma_{\PF}(A, M) = \varnothing$.
Therefore, by \eqref{eq:InclusionChainSpectra} and \eqref{eq:PFClosure}, we get $iI \cap \sigma_{\PF}(A, x) \subseteq iI \cap \sigma_{\PF}(A, \overline{M}) = \varnothing$.
\end{proof}
Lemma \ref{LocalVersionIsolatedPoints} and Proposition \ref{l:LocalPseudofunctionSpectrumIsUnion} yield:
\begin{corollary}
Let $\mathcal{T}$ be a bounded $C_0$-semigroup such that $\sigma_{R}(A)\cap i\R=\varnothing$. Then $\sigma_{\PF}(A)$ is a perfect set.
\end{corollary}

\begin{proof}[Proof of Theorem \ref{t:StableOrbitCountable}]
\quad

    $(a) \Longleftrightarrow (b)$: Apply Theorem \ref{MainResultGregoryJasson} to $\tau(t) =  T(t)x$, whose Laplace transform is $R(\lambda, A)x$.
    
    $(b) \Longrightarrow (c)$: Trivial. 
    
    $(c) \Longleftrightarrow (d) \Longleftrightarrow (e)$: By Lemma \ref{l:Ergodic}.
    
        $(c) \Longrightarrow (b)$: Lemma \ref{LocalVersionIsolatedPoints} shows that the set $\sigma_{\PF}(A, x)$ must be perfect. Since perfect sets are either uncountable or empty \cite{NatansonBook}, we obtain that $\sigma_{\PF}(A, x)=\varnothing.$

        \end{proof}      
        
Let $A$ be the generator of a bounded $C_0$-semigroup. We define its \emph{$L^{1}_{\loc}$-spectrum at $x\in X$} as the complement of the largest open subset set $iI\subseteq i\R$ where $ R(it,A)x \in L^{1}_{\loc}(I;X)$. We write $\sigma_{L^1}(A,x)$ for the $L^{1}_{\loc}$-spectrum of $A$ at $x$. The inclusions $L^{1}_{\loc}(I;X)\subsetneq PF_{\loc}(I;X)$ imply that $\sigma_{\PF}(A,x)\subseteq \sigma_{L^1}(A,x)$, and using the semigroup orbit from Example \ref{ex:PFEmptyButNotContinuous}, one may readily verify that this inclusion can be strict. We immediately obtain the next corollary, which lies intermediately between Theorem \ref{t:StableOrbitCountable} and the Batty-van Neerven-R\"{a}biger local version of the ABLV theorem.   
 
 \begin{corollary}
\label{c:StableOrbitCountable}
Let $\mathcal{T}$ be a bounded $C_0$-semigroup and let $x\in X$. If $\sigma_{L^1}(A, x)$ is countable and $\lim_{\alpha\to0^+} \alpha R(\alpha + i\xi, A) x=0$ for each $i \xi \in \sigma_{L^1}(A, x)$, then the orbit $T(t)x$ is stable.
\end{corollary}

\subsection*{A remark on the boundedness assumption}\label{Boundedness}
According to Remark \ref{r:InclusionsSpectra}{$(v)$}, the local pseudofunction spectrum may be defined even if we relax the assumption of boundedness for the $C_0$-semigroup $\mathcal{T}$ to being polynomially bounded.
However, without the boundedness assumption on the semigroup, the condition $\sigma_{\PF}(A, x) = \varnothing$ alone might fail to deliver the stability of the orbit.
This already follows from an example Arendt and Batty constructed in \cite{A-B-TaubThStabSemigr}. We review the example here for the sake of completeness.

\begin{example}\label{ExampleBoundedness}
\label{ex:BoundCondNess1}
In \cite[Example 2.5(b)]{A-B-TaubThStabSemigr}, the following example is considered:
Let $X = c_0$ and let $\mathcal{T} = (T(t))_{t \geq 0}$ be the $C_0$-semigroup acting on $x = (x_n)_{n \in \N} \in c_0$ as
    \[ (T(t)x)_{2n-1}=e^{-t(1/n^2+in)}[x_{2n-1}+tx_{2n}], \qquad (T(t)x)_{2n}= e^{-t(1/n^2+in)} x_{2n},  \qquad n \in \N . \]
    
Its generator is given by
    \[
        (Ax)_{2n-1} = -\left(\frac{1}{n^2}+in\right)x_{2n-1}+ x_{2n}, \qquad
        (Ax)_{2n} = -\left(\frac{1}{n^2}+in\right)x_{2n} , \qquad
        n \in \N .
    \]

Clearly $\|T(t)\|_{L(X)} \asymp 1 + t$.
But one verifies
	\[ \sigma_{\PF}(A, x) \subseteq\sigma_{\PF}(A)\subseteq \sigma(A)\cap i\R = \varnothing , \qquad x \in X . \]
Let $y \in c_0$ be defined as $y_{2n-1} = 0$ and $y_{2n}=1/n^2$ for $n \in \N$. 
Then, $\|T(t) y\|_X \to e^{-1}$ as $t \to \infty$.
Therefore, $\mathcal{T}$ and the orbit $T(t) y$ are not stable. 
Note that this also shows that even under the extra assumption that the orbit is bounded, the condition $\sigma_{\PF}(A,x)=\varnothing$ is not sufficient to conclude stability.
\end{example}

\section{Characterizations of strong stability}
\label{sec:StrongStabilityChar}
We now turn our attention to strongly stable semigroups. Recall that a $C_0$-semigroup is \emph{strongly stable} if all of its orbits are stable. Notice that strongly stable semigroups are automatically bounded, as follows from the Banach-Steinhaus theorem. 

We will prove the following refinement of Theorem \ref{t:StrongStabCharIntro} stated at the Introduction, which improves upon the ABLV theorem \ref{t:ABLV} and also characterizes strong stability in terms of knowledge of the potential size of $\sigma_{\PF}(A)$ combined with local range conditions.

\begin{theorem}\label{Mainresult}
Let $\mathcal{T}$ be a bounded $C_0$-semigroup. The following statements are equivalent:
	\begin{itemize}
		\item[$(a)$] $\mathcal{T}$ is strongly stable; 
        \item[$(b)$] $\sigma_{\PF}(A,M) = \varnothing$ for some/any dense subset $M \subseteq X$;
         \item[$(c)$] $\bigcup_{x \in M} \sigma_{\PF}(A,x) = \varnothing$ for some dense subset $M \subseteq X$;
		\item[$(d)$] $\bigcup_{x\in M}\sigma_{\PF}(A,x) \subseteq i \R \setminus \sigma_R(A)$ and $\sigma_{\PF}(A, x)$ is countable for any $x$ in some dense subset $M \subseteq X$;
		\item[$(e)$] $\sigma_{\PF}(A)\subseteq i E$ for some 
		 null set $E$ such that, for every compact subset $K \subset \R$,
			\[ \bigcap_{\xi \in E \cap K} \rg(i \xi - A) \text{ is dense in } X ; \]
		\item[$(f)$] 
		for some dense subset $M \subseteq D(A)$, we have $\bigcup_{x \in M} \sigma_{\PF}(A, x) \subseteq iE$, where $E $ is a  
		null set satisfying: For every $\xi_0 \in E$, there exists $\varepsilon = \varepsilon_{\xi_0} > 0$ such that
			\[ \bigcap_{\xi \in E, |\xi - \xi_0| \leq \varepsilon} (i \xi - A)(M) \text{ is dense in } X . \]
	\end{itemize}
\end{theorem}
\begin{proof} Due to the boundedness hypothesis, the semigroup is stable if and only if its orbits $T(t)x$ are stable for every $x$ in a dense subset of $X$. Note also that if the semigroup is stable, one must necessarily have $\sigma_R(A) \cap i \R = \varnothing$, as shown in \cite[Proposition 2.1]{A-B-TaubThStabSemigr}.
Theorem \ref{t:StableOrbitCountable}, \cite[Proposition 2.1]{A-B-TaubThStabSemigr}, and \eqref{eq:InclusionChainSpectra} thus yield the equivalences     $(a) \Longleftrightarrow (b)\Longleftrightarrow (c)\Longleftrightarrow (d)$. The implication $(b) \Longrightarrow (e)$ follows from Proposition \ref{PFClosure},  while $(e) \Longrightarrow (f)$ is trivial.

To complete the proof of the theorem, we will show that $(f) \Longrightarrow (c)$. So assume that the property $(f)$ holds for some fixed dense subset $M\subseteq D(A)$.
Since $E$ is a null set, reducing the size of  $\varepsilon_{\xi_0}$ if necessary, we may assume that $\xi_0 \pm \varepsilon_{\xi_0} \notin E$ for any $\xi_0 \in E$. We may also obviously assume that $E\neq \varnothing$, as otherwise there would be nothing to show.
Put $I_{\xi_0} = (\xi_0 - \varepsilon_{\xi_0}, \xi_0 + \varepsilon_{\xi_0})$ for any $\xi_0 \in E$.
We set
	\[ M_{\xi_0} := \bigcap_{{\xi \in E \cap I_{\xi_0}}} (i\xi - A)(M) , \qquad \xi_0 \in E , \]
which are dense in $X$ by assumption. Employing Corollary \ref{c:SpectrumOf(A-s)}, we deduce that
\begin{equation}
\label{eq1:proofSSS}
\bigcup_{x \in M_{\xi_0}} \sigma_{\PF}(A, x) \subseteq iE.
\end{equation}

We claim that 
	\begin{equation}
		\label{eq:UbetaYbetaClaim}
		i I_{\xi_0} \cap \sigma_{\PF}(A, M_{\xi_0}) = \varnothing , \qquad \forall \xi_0 \in E .
	\end{equation}
Fix $\xi_0\in E$ and $x\in M_{\xi_0}$. 
We need to prove that $R(i\xi,A)x\in \PFloc(I_{\xi_0};X)$. 
For any $\eta \in I_{\xi_0} \cap E$, by writing $x = (i \eta - A)y$ for some $y \in M$, and proceeding exactly as in the proof of Lemma \ref{LocalVersionIsolatedPoints}, we find
 	\begin{equation*} 
		\int_0^t e^{- i \eta s}T(s) x \, \mathrm{d}s
		= y - e^{- i \eta t}T(t) y=O(1) .
	\end{equation*}
Since the inclusion relation \eqref{eq1:proofSSS} yields $\singsupp_{\PF} R(i \, \cdot, A) x \subseteq E$, Theorem \ref{t:SuffCondVanishSingSupp} shows $i I_{\xi_0} \cap \sigma_{\PF}(A, x) = \varnothing $, and the claim \eqref{eq:UbetaYbetaClaim} follows.

Proposition \ref{PFClosure} and \eqref{eq:UbetaYbetaClaim} now yield $\sigma_{PF}(A)\cap i E=\varnothing$. Combining the latter with \eqref{eq1:proofSSS}, we have verified that property $(c)$ holds with any of the dense subspaces $M_{\xi_{0}}$.
\end{proof}

The following special case of Theorem \ref{Mainresult} follows at once (cf. Corollary \ref{c:StableOrbitCountable}):

\begin{corollary}\label{c:Mainresult}
Let $\mathcal{T}$ be a bounded $C_0$-semigroup. Any of the following conditions is sufficient for the stability of $\mathcal{T}$:
	\begin{itemize}
		\item[$(a)$]$\bigcup_{x\in M}\sigma_{L^1}(A,x) \subseteq i \R \setminus \sigma_R(A)$ and $\sigma_{L^1}(A, x)$ is countable for any $x$ in some dense subset $M \subseteq X$; 	
		\item[$(b)$] 
		for some dense subset $M \subseteq D(A)$, we have $\bigcup_{x \in M} \sigma_{L^1}(A, x) \subseteq iE$, where $E $ is a  
		null set satisfying: For every $\xi_0 \in E$, there exists $\varepsilon = \varepsilon_{\xi_0} > 0$ such that
			\[ \bigcap_{\xi \in E, |\xi - \xi_0| \leq \varepsilon} (i \xi - A)(M) \text{ is dense in } X . \]
	\end{itemize}
\end{corollary}

\begin{remark}
\label{a practical remark} In practice, a first potential candidate to play the role of the null set from Theorem \ref{Mainresult} might be $E=\sigma(A)\cap i\R$ itself. This leads to Corollary \ref{c: ss}, or yet a more localized version where, similarly to $(f)$ of Theorem \ref{Mainresult}, one can replace \eqref{eq: local range ibs} by: for each $i\xi_0\in \sigma(A)\cap i\R$ there is an open interval $\xi_0\in I_{\xi_0}$ such that 
\[		 \bigcap_{\lambda \in iI_{\xi_0}\cap\,\sigma(A)} \rg(\lambda - A) \text{ is dense in } X. \] 
\end{remark}

\begin{remark}\label{ExampleConditionMainResult} 
As already mentioned at the Introduction, Arendt and Batty argued in \cite[Example 2.5(a)]{A-B-TaubThStabSemigr} that the ABLV theorem \ref{t:ABLV} was sort of optimal in the sense that for any closed uncountable set $E \subseteq \R$ there exists a bounded \emph{non strongly stable} $C_0$-semigroup $\mathcal{T}$ for which $\sigma_R(A) \cap i \R = \varnothing$ and $\sigma(A) \cap i \R \subseteq i E$. 
To clarify that their example is not in contradiction with Theorem \ref{Mainresult} and Corollary \ref{c: ss}, we demonstrate that condition Theorem \ref{Mainresult}(e) is actually not met in that particular instance. The key point in our extensions from Theorem \ref{Mainresult} for null sets is that we have found an adequate range condition for removing the local pseudofunction singularities of the resolvent from the imaginary axis.

Let $X = L^2( w, \mu)$ and $\mathcal{T}$ be as in Example \ref{ex:UnionNotPseudoSpectrum}, where we additionally assume that the closed set $ w$ is uncountable and $\mu$ is non-atomic.
Since it is an isometry semigroup, it is not strongly stable. Let $E \subseteq  w$ be a closed subset.

Suppose that, for every compact subset $K \subseteq \R$, we have
	\[ X_K := \bigcap_{\xi \in E \cap K} \rg(i \xi - A) \text{ is dense in } L^2( w, \mu) . \]
Then $E$ is necessarily a $\mu$-null set.
Indeed, otherwise, the set $E'$ of all non-isolated points of $E$ is not a $\mu$-null set, as $\mu$ is non-atomic.
But then also $\mu(E' \cap K) > 0$ for some compact $K \subseteq \R$, and we may assume $E' \cap K$ does not have isolated points.
For any $f$ beloning to $X_K$ and for each $\xi \in E \cap K$, there exists a $g_\xi \in D(A)$ such that $f(u) = (i \xi - u) g_\xi(u)$.
This leads to $\text{ess}\supp(f) \cap (E' \cap K) = \varnothing$.
But then $X_K$ cannot possibly be dense in $L^2( w, \mu)$.
Therefore, we see that the property $(e)$ from Theorem \ref{Mainresult} can only hold if $\mu(E) = 0$.

On the other hand, if $\sigma_{\PF}(A, f) \subseteq i E$ for each $f \in L^2( w, \mu)$ (or even a dense subset), then, by \eqref{eq:PFSpec=EssSupp}, it must necessarily hold that $\mu( w \setminus E) = 0$.
In particular, $E$ cannot be a $\mu$-null set.
\end{remark}

\section{An example: the transport equation}
\label{sec:TransportEq}

We consider a family of linear PDE Cauchy problems for which our results can be used to conclude stability of all solutions.
In contrast, neither Theorem \ref{t:ABLV} nor \ref{t:FourierTypeRangeThm} is applicable in general.
In particular, the computations below highlight that, in practice, it may be possible to determine $\sigma_{\PF}(A)$ via direct control on $R(\lambda, A)$, without requiring the explicit determination of the semigroup.

We will consider the transport equation
	\begin{equation} 
		\label{eq:TransportEq}
		\frac{\partial u}{\partial t} + a(s)\frac{\partial u}{\partial s} = 0 , \qquad (t, s) \in \R_+ \times \R_+ .
	\end{equation}
Here $a : \R_+ \to (0, \infty)$ is bounded and measurable, so that
	\begin{equation*}
		\label{eq:PhiDef}
		\Phi(s) = \int_0^s \frac{\mathrm{d}u}{a(u)} 
	\end{equation*}
defines a continuous bijection $\R_+ \to \R_+$. 
We let the initial data be an element of the weighted Banach space $X = L^1(\R_+, w(s)\mathrm{d}s)$, the space introduced in Example \ref{ex:PseudoZeroUnitaryWholeLine}. 
For the time being, we only assume that $w:\R_+\to (0,\infty)$ is measurable.
Then, \eqref{eq:TransportEq} takes the form of \eqref{eq:ACP} for the closed operator $Af= -af'$ with dense domain $D(A) = \{ f \in X \mid af' \in X \text{ and } f(0) = 0 \}$.

When $a \equiv 1$, we revert to Example \ref{ex:PseudoZeroUnitaryWholeLine}, and the solutions to \eqref{eq:TransportEq} are stable.
However, we have seen that $\sigma_u(A, f) = i \R$ for all $f \in X \setminus \{0\}$.
Therefore, the ABLV theorem \ref{t:ABLV}, or its extension to individual orbits as in \cite[Theorem 3.4]{B-V-R-LocalSpecInStab}, cannot yield the stability conclusion.
Additionally, for the specific choice of $w(s) = (\log(e + s))^{-1}$, \cite[Proposition 4.3]{C-T-StabilityGeometry} shows that also Theorem \ref{t:OriginalRangeThm} fails to apply here. 

The aim of this section is to demonstrate that for more general $a$, through corresponding conditions on $w$, we may still conclude stability of the solutions to \eqref{eq:TransportEq} using our results.
To do this, we will rely only on the operator $A$ to illustrate that our characterizations are applicable in concrete situations, where, in fact, the local pseudofunction spectrum is explicitly computable.
We also show how the classical results may fail to conclude strong stability.
For the sake of completeness, we remark that the semigroup\footnote{We can of course also show strong stability from this explicit formula for the semigroup, but our point here is once again that explicit knowledge of the semigroup is not needed to carry out an analysis of $\sigma_{\PF}(A)$.} solving \eqref{eq:TransportEq}, with initial data $f$, is given by 
	\[
		(T(t)f)(s)
		=
		\begin{cases}
			(f \circ \Phi^{-1})(\Phi(s) - t) , & s \geq \Phi^{-1}(t) , \\
			0 , & s < \Phi^{-1}(t) . 
		\end{cases}
	\]

For any $f \in X$, by solving the differential equation $\lambda g + a g' = f$, we see that
	\[ (R(\lambda,A) f)(s) = e^{- \lambda \Phi(s)} \int_0^s e^{\lambda \Phi(u)} \frac{f(u)}{a(u)} \, \mathrm{d}u , \qquad \Real{\lambda} > 0 . \] 
Of course, this identity is only true when the right-hand side is an element of $D(A)$.
This will be the case when $aw$ is non-increasing.
Moreover, when $aw$ vanishes at infinity, we may already show strong stability.
	
\begin{proposition}\label{p:stabletransport}
Let $a w$ be non-increasing.
Then $A$ generates a bounded $C_0$-semigroup.
If, additionally, $aw$ vanishes at infinity,
 then the solutions to \eqref{eq:TransportEq} with initial data in $X$ are stable.
\end{proposition}

\begin{proof}
We start by verifying that $A$ generates a bounded $C_0$-semigroup.
Since $aw$ is non-increasing, for any $f \in X$,
	\begin{align*}
			\Vert R(\lambda,A)f\Vert_{X}
			&=\int^\infty_0 \left|  e^{-\lambda \Phi(s)} \int_0^s e^{\lambda \Phi(u)} \frac{f(u)}{a(u)} \, \mathrm{d}u \right| \, w(s) \mathrm{d}s
			\\
			&\leq \int^\infty_0 \left( \int_u^\infty \frac{e^{-(\Real{\lambda}) \Phi(s)}}{a(s)} \, \mathrm{d}s \right) e^{(\Real{\lambda}) \Phi(u)} |f(u)| \, w(u) \mathrm{d}u 
			\\
			&=\int^\infty_0 \left( \int_{\Phi(u)}^\infty e^{-(\Real{\lambda}) v} \, \mathrm{d}v \right) e^{(\Real{\lambda}) \Phi(u)} |f(u)| \,  w(u) \mathrm{d}u 
			=\frac{\|f\|_{X}}{\Real{\lambda}}.
	\end{align*}
Therefore, $\C_+ \subseteq \rho(A)$ and
	\[ \Vert R(\lambda,A)f\Vert_{L(X)} \leq \frac{1}{\Real{\lambda}} , \qquad \Real{\lambda} > 0 . \]
In particular, $A$ generates a contractive, and thus bounded, $C_0$-semigroup by the Hille-Yosida generation theorem \cite[Theorem 3.5, p.~66]{E-N-ShortCourseSemigroups}.

Suppose now that $aw$ vanishes at infinity.
We show that $\sigma_{\PF}(A, f) = \varnothing$ for any $f \in \D(0, \infty)$.
Then $A$ would generate a strongly stable $C_0$-semigroup by Theorem \ref{Mainresult}, as $\D(0, \infty)$ is dense in $X$.
Fix some $f \in \D(0, \infty)$.
Then, distributionally, 
	\[ (R(i \xi, A) f)(s) = e^{- i \xi \Phi(s)} \int_0^s e^{i \xi \Phi(u)} \frac{f(u)}{a(u)} \, \mathrm{d}u  \in \mathcal{D}'(\R; X). \]
Take arbitrary $\varphi \in \D(\R)$.
For any $\varepsilon > 0$, choose $s_\varepsilon > 0$ large enough such that $ a(s)w(s) \leq \varepsilon$ when $s \geq s_\varepsilon$.
As $aw$ is non-increasing, we find
	\begin{align*}
		\Vert \langle & R(i\xi,A)f,e^{i t \xi}\varphi(\xi) \rangle \Vert_{X}	
		\\& \quad	=\int^\infty_0 \left| \int_{-\infty}^\infty \left( e^{-i\xi \Phi(s)}\int^s_0 e^{i\xi \Phi(u)} \frac{f(u)}{a(u)} \mathrm{d}u \right) e^{i t \xi}\varphi(\xi) \, \mathrm{d}\xi \right| \,  w(s) \mathrm{d}s 
		\\& \quad
		=\int^\infty_0 \left| \int_{0}^s \frac{f(u)}{a(u)} \widehat\varphi\Big(\Phi(s)-\Phi(u)-t\Big) \, \mathrm{d}u \right|  \,  w(s) \mathrm{d}s 
		\\&\quad
			=\int^{\Phi(s_{\varepsilon})}_{0} \left| \int_{0}^v f(\Phi^{-1}(u))\widehat\varphi\Big(v-u-t\Big) \, \mathrm{d}u \right|   a(\Phi^{-1}(v))w(\Phi^{-1}(v)) \mathrm{d}v 
				\\&\quad\qquad
		+ \varepsilon\int^\infty_{\Phi(s_\varepsilon )}\left|
		 \int_{0}^v f(\Phi^{-1}(u)) 
		\widehat\varphi\Big(v-u-t\Big) \, \mathrm{d}u 
		\right| \mathrm{d}v 
			\\&\quad
		\leq a(0)w(0)\int_{0}^{\Phi(s_\varepsilon)} \left( \int_{-t}^{ - t +u+\Phi(s_\varepsilon)} |\widehat\varphi(v)| \, \mathrm{d}v \right) |f(\Phi^{-1}(u))|\, \mathrm{d}u + \varepsilon\Vert\hat{\varphi}\|_{L^1(\R)} \|f / a\|_{L^1(\R_+)} 
		\\&\quad
		=o_{\varepsilon}(1)+\varepsilon\Vert\hat{\varphi}\|_{L^1(\R)} \|f / a\|_{L^1(\R_+)} , \qquad |t| \to\infty.
	\end{align*}
Taking the limit superior $|t|\to\infty$ first and then $\varepsilon\to 0^{+}$, 
we conclude
	\[ \| \langle R(i\xi,A)f, e^{i t \xi}\varphi(\xi) \rangle \|_{X} = o(1) , \qquad |t| \to \infty . \]
Therefore, $\sigma_{\PF}(A, f) = \varnothing$.
\end{proof}

 We now further investigate the spectrum by specifying the decay of $aw$ more. 
Through the latter, we may conclude whether or not certain eigenfunctions of $A^*$ belong to $X^*$.
Note that $X^* = L^\infty(\R_+, \frac{\mathrm{d}s}{w(s)})$. 
If $aw$ is non-increasing, using integration by parts, we see that $A^* g = (a g)'$ with domain 
\[D(A^*) = \{ h \in X^* \mid (ah)' \in X^* \text{ and } \lim_{s\to \infty} a(s)h(s)f(s) = 0,\: \forall f\in D(A)   \}.\] 
Formally, the eigenspace of  $A^*$ with eigenvalue $\lambda \in \C$ is spanned by the function
	\begin{equation}
		\label{eq:Eigenfunctions} 
		h_\lambda(s) =   \frac{1}{a(s)} e^{\lambda \Phi(s)};
	\end{equation}
we automatically obtain $h_\lambda \notin X^*$ for $\Real{\lambda} > 0$ when $aw$ is bounded.

We now start by imposing minimal lower  growth bounds for $aw$.
First, we consider lower bounds that guarantee $h_\lambda \in D(A^*)$ for every $\Real{\lambda} < 0$.
This leads to the spectrum containing the imaginary axis.

\begin{proposition}
Let $aw$ be non-increasing and
	\begin{equation} 
		\label{eq:CondWeightsForUniformSpec} 
		\sup_{s \in \R_+} \frac{e^{- \alpha \Phi(s)}}{a(s) w(s)} < \infty , \qquad \forall \alpha > 0 . 
	\end{equation}
Then $\sigma_u(A, f) = i \R$ for any $f \in X \setminus \{0\}$.
In particular, $\sigma(A) \cap i \R = i \R$.
\end{proposition}

\begin{proof}
We adapt the proof of \cite[Example 4.1]{B-V-R-LocalSpecInStab}.
By our assumption on $aw$, we have $h_\lambda \in D(A^*)$ for $\Real{\lambda} < 0$.
Suppose now $f \in X$ is such that $i \eta \notin \sigma_u(A, f)$.
Then, there is a connected neighborhood $V$ of $i \eta$ where $R(\lambda, A)f$ has a holomorphic extension $F$ on $V$.
For any $\lambda \in V$ with $\Real{\lambda} > 0$ we have, by the resolvent identity,
	\[ R(1,A)f = F(\lambda)+ (\lambda-1)R(1,A)F(\lambda) = (\lambda-A)R(1,A)F(\lambda) . \] 
This identity then holds for any $\lambda \in V$.
Now, for any $\lambda \in V$ with $\Real{\lambda} < 0$,
	\[ \int^\infty_0 h_\lambda(s)(R(1,A)f)s) \mathrm{d}s = \langle h_\lambda, R(1,A)f\rangle = \langle(\lambda-A^*)h_\lambda, R(1,A)F(\lambda)\rangle = 0 . \]
The left-hand side is an analytic function on $\Real{\lambda} < 0$.
So, necessarily, for any $\Real{\lambda} < 0$,
	\[ 0 = \int_0^\infty e^{\lambda \Phi(s)} (R(1, A) f)(s) \, \frac{\mathrm{d}s}{a(s)} = \int_0^\infty e^{\lambda v} (R(1, A) f)(\Phi^{-1}(v)) \, \mathrm{d}v . \]
By the uniqueness of the Laplace transform, we conclude that $(R(1, A) f) \circ \Phi^{-1} \equiv 0$.
As $\Phi$ is a bijection and $R(1, A)$ is injective, we must have $f \equiv 0$.
\end{proof}

By the previous results, if $aw$ is non-increasing, vanishes at infinity, and if \eqref{eq:CondWeightsForUniformSpec} holds, then $A$ generates a strongly stable $C_0$-semigroup, but the ABLV theorem \ref{t:ABLV} cannot be used to conclude this. 
Moreover, \cite[Theorem 3.4]{B-V-R-LocalSpecInStab} does not show the stability of any non-trivial orbit.
Both failures do not stem from singularities of the residual spectrum, since strong stability ensures that $\sigma_R(A) \cap i \R = \varnothing$.
Note that \eqref{eq:CondWeightsForUniformSpec} holds if, for example,
	\[ \sup_{s \in \R_+} \frac{e^{- \alpha s}}{a(s) w(s)} < \infty , \qquad \forall \alpha > 0 .  \]

By imposing finer lower growth bounds for $aw$ than in \eqref{eq:CondWeightsForUniformSpec}, we may also show the triviality of the intersection of ranges.

\begin{proposition}
\label{p:RangeConditionFailure}
Let $a$ and $w$ be such that $aw$ is non-increasing, vanishes at infinity, and
	\begin{equation}
		\label{eq:CondWeightsForRangeCond} 
		\sup_{s \in \R_+} \frac{e^{- \alpha \Phi(s)}}{a(s) w(s)} = o\left(\frac{1}{\alpha}\right) , \qquad \alpha \to 0^+ . 
	\end{equation}
Then, for any non-empty open interval $I \subseteq \R$, we have 
	\[ \bigcap_{\xi \in I} \, \rg(i \xi - A) = \left\{0\right\}.
	 \]
\end{proposition}

To prove Proposition \ref{p:RangeConditionFailure}, we will need the following extension of \cite[Lemma 4.2]{C-T-StabilityGeometry}.

\begin{lemma}
\label{l:LaplaceCompositionPhi-1}
Suppose \eqref{eq:CondWeightsForRangeCond} holds.
If $g \in X$, then, uniformly for $\xi \in \R$,
	\[ \mathcal{L}\{g \circ \Phi^{-1}; \alpha + i \xi\} = o\left(\frac{1}{\alpha}\right) , \qquad \alpha \to 0^+ . \]
\end{lemma}

\begin{proof}
For $\xi \in \R$,
	\begin{align*}
		\left| \alpha \mathcal{L}\{g \circ \Phi^{-1}; \alpha + i \xi\} \right|
		&=\left| \alpha\int^\infty_0 e^{-(\alpha + i\xi)s}g(\Phi^{-1}(s)) \, \mathrm{d}s \right| \\
		&=\left| \alpha\int^\infty_0 e^{-(\alpha + i\xi)\Phi(s)} \frac{g(s)}{a(s)  w(s)} \,  w(s)\mathrm{d}s \right| \\		
		&\leq \alpha \|g\|_X \sup_{s \in \R_+}\frac{e^{-\alpha \Phi(s)}}{a(s) w(s)}
		= o(1) . \qedhere
	\end{align*}		
\end{proof}

\begin{proof}[Proof of Proposition \ref{p:RangeConditionFailure}]
Our proof follows that of \cite[Proposition 4.3]{C-T-StabilityGeometry}.
Fix some non-empty open interval $I \subseteq \R$.
Suppose $f \in \bigcap_{\xi \in I} \rg(i \xi - A)$.
By Proposition \ref{p:stabletransport}, we know that $A$ generates a strongly stable $C_0$-semigroup.
Fix $\xi \in I$ and let $g_\xi \in D(A)$ be such that $f = (i \xi - A) g_\xi$.
Since $\sigma_R(A) \cap i \R = \varnothing$,  we  have $g_\xi \in \overline{\rg(i \xi - A)}$.
Then, Lemma \ref{l:Ergodic} shows
	\[ \lim_{\alpha \to 0^+} R(\alpha + i\xi, A) f = g_\xi + \lim_{\alpha \to 0^+} \alpha R(\alpha + i \xi, A) g_\xi = g_\xi . \]
Therefore, for almost all $s \in \R_+$,
	\[ g_\xi(s) = e^{- i \xi \Phi(s)} \int_0^s e^{i \xi \Phi(u)} \frac{f(u)}{a(u)} \, \mathrm{d}u=e^{-i \xi \Phi(s)} \int_0^{\Phi(s)} e^{i \xi v}f(\Phi^{-1}(v)) \, \mathrm{d}v . \] 
Hence, for almost all $s \in \R_+$,
	\[ g_\xi(\Phi^{-1}(s)) = e^{-i \xi s} \int^s_0 e^{i \xi v} f(\Phi^{-1}(v)) \, \mathrm{d}v . \]
Taking Laplace transforms gives
	\[ \mathcal{L}\{ g_\xi \circ \Phi^{-1}; \alpha - i \xi \} = \alpha^{-1} \mathcal{L}\{ f \circ \Phi^{-1} ; \alpha - i \xi \} , \qquad \alpha > 0 . \]
Lemma \ref{l:LaplaceCompositionPhi-1} shows
	\[ \lim_{\alpha \to 0^+} \mathcal{L}\{ f \circ \Phi^{-1} ; \alpha - i \xi \} = 0 , \qquad \xi \in I . \]
By the Privalov uniqueness theorem \cite[Corollary 6.14, p.~140]{Pommerenke}, this can only happen when $f \equiv 0$.
\end{proof}

If now $aw$ is non-increasing, tends to zero at infinity, and \eqref{eq:CondWeightsForRangeCond} holds, then $A$ generates a strongly stable $C_0$-semigroup, but this can not be concluded from Theorems \ref{t:ABLV}--\ref{t:FourierTypeRangeThm}. 
We note that \eqref{eq:CondWeightsForRangeCond} holds if, for example,
	\[ \sup_{s \geq 1} \frac{s^{-\rho}}{a(s) w(s)} <\infty , \qquad \text{for some } \rho \in (0, 1) . \]
	
\section{Semi-uniform stability and the uniform pseudofunction spectrum}
\label{sec:SemiUniform}

Up to this point, we have exclusively considered the boundary behavior of $R(\lambda, A)x$ for fixed $x \in X$.
However, the resolvent of a bounded $C_0$-semigroup has itself $L(X)$-valued distributional boundary values on the whole imaginary line. It is then also natural to introduce the following definition.

\begin{definition}
\label{def:UniformPseudofuncSpec}
The \emph{uniform pseudofunction spectrum $\sigma_{\uPF}(A)$} of a bounded $C_0$-semigroup $\mathcal{T}$ is given by
	\[ \sigma_{\uPF}(A) = i \singsupp_{\PF} R(i \, \cdot, A) . \]
\end{definition}

It is immediate that $\sigma_{\uPF}(A) \subseteq \sigma(A) \cap i \R$. The main result of this section is that one has, in fact, equality.

\begin{proposition}\label{SpectraCoincide}
	Let $\mathcal{T}$ be a bounded $C_0$-semigroup. It holds that 
	\[\sigma_{\uPF}(A)=\sigma(A)\cap i\R.\]
\end{proposition}

To prove this proposition, we need the following lemma.

\begin{lemma}
	Let $Z$ be a Banach space and let $I \subseteq \R$ be an open interval. 
	Let $F : (0,\alpha_0) + iI \to Z$ be holomorphic. 
	If $F$ has local pseudofunction boundary behavior on $iI$, then, for each compact subset $K\subseteq I$,
		\begin{equation}
			\label{eq:LocalPseudoBoundBehImpliesSmallBlowup} 
			F(\alpha+ i \xi) =o\left( \frac{1}{\alpha} \right), \qquad \alpha \to 0^+ ,  
		\end{equation}
uniformly for $ \xi \in K$. 
\end{lemma}

\begin{proof}
	Let $f(\xi)=F(i\xi)\in \D'(I,Z)$ be the distributional boundary value of $F$ on $i I$. 
	Fix some compact $K \subseteq I$.
	There exists an open set $I' \subseteq I$ containing $K$ such that $f = \widehat{g}$ on $I'$ for some $g \in C_{0}(\R; Z)$. 
	Define the $Z$-valued function 
	\[
		U(\lambda)=\int_{-\infty}^0 e^{\overline{\lambda}t} g(t) \, \mathrm{d}t + \int^{\infty}_0 e^{-\lambda t} g(t) \, \mathrm{d}t
		 , \qquad \lambda \in \C_+ . 
	\]
	Then $U$ is harmonic on $\C_+$.
	A small calculation using $g(t)=o(1)$ shows
	\begin{equation}
		\label{FunctUislittleO}
		\lim_{\alpha \to 0+} \alpha \cdot \sup_{\xi \in\R} \|U(\alpha+i\xi)\|_Z = 0 .
	\end{equation}
One has (see, for example, \cite{B-DistrComplVarFourTrans})
		\[ \widehat{g}(\xi) = \lim_{\alpha \to 0^+} U(\alpha + i\xi)  \]
	in $\D^\prime(I;Z)$.
	Applying a $Z$-valued version of the distributional reflection principle for harmonic functions \cite{B-DistrComplVarFourTrans,R-EdgeWedge}, we conclude that $F - U$ has harmonic extension to a complex neighborhood of $i I'$. In particular, $F(\lambda)=U(\lambda)+O(1)$ on $\lambda\in (0,\alpha_0] + iK$. Then \eqref{eq:LocalPseudoBoundBehImpliesSmallBlowup} follows from \eqref{FunctUislittleO}.
\end{proof}

\begin{proof}[Proof of Proposition \ref{SpectraCoincide}]
Suppose $i \xi_0 \notin \sigma_{\uPF}(A)$.
Considering a rescaled semigroup, we may assume $\xi_0 = 0$.
Then $R(i \xi, A) \in \PFloc(I; Z)$ for some open interval $I$ containing $0$.
Using \eqref{eq:LocalPseudoBoundBehImpliesSmallBlowup} for $F(\lambda) = R(\lambda, A)$, we obtain $\alpha \|R(\alpha, A)\|_{L(X)} < 1$ for certain small enough $\alpha > 0$.
But then $0 \in \rho(A)$, as follows from the series representation of the resolvent.
We have thus shown $\sigma(A) \cap i \R \subseteq \sigma_{\uPF}(A)$, which completes the proof of the proposition.
\end{proof}

Observe that $T(t)$ is not necessarily (uniformly) continuous as an $L(X)$-valued function. 
Theorem \ref{MainResultGregoryJasson} is therefore generally not applicable to this function.
Consequently, we cannot directly conclude uniform asymptotic behavior for the semigroup $\mathcal{T}$ from its uniform pseudofunction spectrum $\sigma_{\uPF}(A)$.
However, we may consider semi-uniform stability.

Recall \cite{B-D-NonUniformStab,C-S-T-SemiUniformStabOpSemigroupEnergyDecay} a bounded $C_0$-semigroup $\mathcal{T}$ is called \emph{semi-uniformly stable} if $0 \notin \sigma(A)$ and
	\[ \lim_{t \to \infty} \| T(t) A^{-1} \|_{L(X)} = 0 . \]
As $D(A)$ is dense and $\mathcal{T}$ is bounded, this in particular implies that $\mathcal{T}$ is strongly stable.

Applying Proposition \ref{SpectraCoincide} and then Lemma \ref{l:SingSuppEmpty<=>o(1)} to the $L(X)$-valued Lipschitz continuous function $\tau(t)=T(t)A^{-1}$, one recovers the following characterization of semi-uniform stability that goes back to \cite{A-B-TaubThStabSemigr} (cf. \cite[Theorem 1.1]{B-D-NonUniformStab}).

\begin{corollary}
	\label{t:OurSemiUniformSpecChar}
	A bounded $C_0$-semigroup $\mathcal{T}$ is semi-uniformly stable if and only if 
	\[\sigma(A)\cap i\R =  \sigma_{\uPF}(A) = \varnothing.\]
\end{corollary}

\section{Tauberian theorems of Katznelson-Tzafriri type}
\label{sec:KatznelsonTzafriri}

For a bounded $C_0$-semigroup $\mathcal{T}$ and a function $f \in L^1(\R)$, we can define the following continuous linear operator on $X$,
	\[ \widehat{f}(T) = \int_{0}^\infty f(t) T(t) \, \mathrm{d}t  \]
where the integral is taken in the strong operator topology. 
Note that the linear map
	\[ L^1(\R) \to L(X) , \qquad f \mapsto \widehat{f}(T), \]
is continuous.
We are interested in the stability of $\mathcal{T}$ on the range of $\widehat{f}(T)$ when $f$ vanishes on $(-\infty,0)$.
In particular, we will show orbital/uniform stability if the spectrum\footnote{If $f\in\mathcal{S}'(\R)$, one often refers to $\operatorname*{spec}(f)=\supp \widehat{f}$ as its spectrum.} of $f$ sufficiently cancels the local/uniform pseudofunction spectrum of $\mathcal{T}$.

Observe
	\[ T(t) \widehat{f}(T) = \widehat{f(\cdot - t)}(T) , \qquad f \in L^1(\R_+), ~ t \geq 0 . \]
On the other hand, if $f = \widehat{\varphi}$ for some $\varphi \in \D(\R)$, we clearly have
	\[
		\widehat{f}(T) x = \ev{R(i \xi, A) x}{\varphi(\xi)} , \qquad x \in X . 
	\]
Now, if $f \in L^1(\R_+)$ and if $(\varphi_n)_{n \in \N}$ is a sequence in $\D(\R)$ such that $\widehat{\varphi}_n \to f$ in $L^{1}(\R)$, by combining all previous observations, we see that
	\begin{equation}
		\label{eq:f(T)Resolvent}  
		T(t) \widehat{f}(T) x = \lim_{n \to \infty} \ev{R(i \xi, A) x}{e^{i t \xi} \varphi_n(\xi)} , 
	\end{equation}
uniformly in both $t \geq 0$ and $x$ in bounded subsets of $X$. This naturally connects to the notion of spectral synthesis.

\begin{definition}
\label{def:SpectralSynthesis}
Let $E \subseteq \R$ be a closed subset.
A function $f \in L^1(\R_+)$ is of \emph{spectral synthesis with respect to $E$} if there exists a sequence $(f_n)_{n \in \N}$ in $L^1(\R)$ such that each $\widehat{f}_n$ vanishes on a neighborhood of $E$ and $f_n \to f$ in $L^1(\R)$. 
\end{definition}

\begin{remark}
If a non-trivial $f \in L^1(\R_+)$ is of spectral synthesis with respect to some closed
 set $E$, then in particular $\widehat{f}(\xi) = 0$ for any $\xi \in E$.
The well-known theorem of F. Riesz and M. Riesz \cite{R-R-RandwerteAnalFunk} yields that $E$ must have Lebesgue measure zero. 
On the other hand, Malliavin \cite{M-ImposSynthSpec} has shown the existence of a closed null set $E$ and a non-trivial function $f \in L^1(\R_+)$ with $\widehat{f}$ vanishing on $E$ but $f$ not being of spectral synthesis with respect to $E$, see also \cite[Chapter VIII, Section 7]{K-IntroHarmAnal}. 
Furthermore, one can show there are closed null sets $E$ for which no non-trivial element of $L^1(\R_+)$ is of spectral synthesis with respect to it; see Proposition \ref{p:NoSpectralSynthesis} below.
\end{remark}

We note that the approximation sequence in the definition of spectral synthesis can always be taken from $\mathcal{F} (\D(\R))$. We leave the verification of the following simple lemma to the reader.

\begin{lemma}
\label{l:SpecSynthRegSeq}
If $f \in L^1(\R_+)$ is of spectral synthesis with respect to a closed set $E \subseteq \R$, then there exists a sequence $(\phi_n)_{n \in \N}$ in $\mathcal{F}(\D(\R))$ such that each $\widehat{\phi}_n$ vanishes on a neighborhood of $ E$ and $\phi_n \to f$ in $L^1(\R)$.
\end{lemma}

We now present two Tauberian results within the next theorem.
In fact, they both turn out to characterize spectral synthesis, thereby showing that the results are, in general, optimal. Given a function $g$, we define its reflection about the origin $\check{g}$ as $\check{g}(t)=g(-t)$.

\begin{theorem}
\label{t:KatznelsonTzafriri}
Let $f \in L^1(\R_+)$, and let $E$ be a closed
 null set.
	The following statements are equivalent:
		\begin{itemize}
			\item[$(a)$] $f$ is of spectral synthesis with respect to $E$;
			\item[$(b)$] for any bounded $C_0$-semigroup $\mathcal{T}$ and $x \in X$,
				\[ i \sigma_{\PF}(A, x) \subseteq E \quad \Longrightarrow \quad  \lim_{t \to \infty} \|T(t) \widehat{f}(T) x \|_X = 0 ; \]
			\item[$(c)$] for any bounded $C_0$-semigroup $\mathcal{T}$,
				\[ 
				i \sigma(A)\cap \R \subseteq E \quad \Longrightarrow \quad \lim_{t \to \infty} \| T(t) \widehat{f}(T) \|_{L(X)} = 0. \] 
						\end{itemize}
\end{theorem}

\begin{proof}
If $f$ is of spectral synthesis with respect to $E$, then $(b)$ and $(c)$ follow directly from \eqref{eq:f(T)Resolvent} and Lemma \ref{l:SpecSynthRegSeq}.

Suppose now $f$ is not of spectral synthesis with respect to $E$.
We show that $(b)$ and $(c)$ cannot hold, following an idea due to Bourgain and used in \cite{K-T-PowerBoundOp}.
We first construct a suitable bounded uniformly continuous function $\kappa$.
Consider
	\[ \mathcal{I}(E) = \{ g \in L^1(\R) \mid \widehat{g} \text{ vanishes on a neighborhood of } E \} . \]
Then $f \notin \overline{\mathcal{I}(E)}$.
Consequently, by the Hahn-Banach theorem, there exists a $\upsilon \in L^\infty(\R)$ such that 
	\[ \int_{0}^\infty f(u) \upsilon(- u) \,  \mathrm{d}u  \neq 0\quad \text{but} \quad \int_{-\infty}^\infty g(u) \upsilon(- u) \, \mathrm{d}u = 0, ~ \forall g \in \mathcal{I}(E) . \]
Next, we take any $\chi \in L^1(\R)$ such that $\int_{-\infty}^\infty \chi(t)  \mathrm{d}t = 1$ and set $\chi_n(t) = n \chi(n t)$ and $\upsilon_n = \chi_n * \upsilon$ for $n \in \N$.
Note that each $\upsilon_n$ is bounded uniformly continuous.
Since
	\[ \int_0^\infty f(u) \upsilon_n(-u) \, \mathrm{d}u = [f * \upsilon_n](0) = [(f * \chi_n) * \upsilon](0) \to [f * \upsilon](0) = \int_0^\infty f(u) \upsilon(-u) \, \mathrm{d}u \neq 0 , \]
as $n \to \infty$, there exists some $n_0 \in \N$ such that $c := \int_0^\infty f(u) \upsilon_{n_0}(-u) \mathrm{d}u \neq 0$.
Put $\kappa = \upsilon_{n_0} / c$.
As
 $g * \chi_{n_0} \in \mathcal{I}(E)$ for every $g \in \mathcal{I}(E)$, we see that $\kappa$ satisfies
	\[ \int_{0}^\infty f(u) \kappa(- u) \,  \mathrm{d}u = 1 \quad \text{but} \quad \int_{-\infty}^\infty g(u) \kappa(- u) \, \mathrm{d}u = 0, ~ \forall g \in \mathcal{I}(E) . \]

Necessarily, $\supp \widehat{\kappa} \subseteq E$.
Indeed, if $\xi_0 \notin E$, then $\xi_0 \in I$ for some open $I \subseteq \R$ such that $E \cap \overline{I} = \varnothing$.
For any $\varphi \in \D(I)$
 $\widehat{\varphi}(-\xi) \in \mathcal{I}(E)$.
Therefore, $\ev{\widehat{\kappa}}{\varphi} = 0$ for any $\varphi \in \D(I)$, so that $\supp \widehat{\kappa} \cap I = \varnothing$, which shows our claim.

Let $BUC(\R)$ denote the space of all bounded uniformly continuous functions on $\R$.
We will work with the Banach space
	\[ \widehat{BUC}(\R) = \{ \psi \in \S^\prime(\R) \mid \widehat{\psi} \in BUC(\R) \} , \qquad \|\psi\|_{\widehat{BUC}} = \|\widehat{\psi} \|_{L^\infty} , \]
	and the bounded $C_0$-semigroup $\mathcal{T}$ defined on $\widehat{BUC}(\R)$ as $(T(t) \psi)(u) = e^{- i t u} \psi(u)$.
Note that 
	\[ \|T(t) \psi\|_{\widehat{BUC}} = \|\widehat{\psi}(\cdot + t)\|_{L^\infty} = \|\widehat{\psi}\|_{L^\infty} = \|\psi\|_{\widehat{BUC}} , \qquad \psi \in \widehat{BUC}(\R), ~t \geq 0 . \]
Therefore, the orbit $T(t)\psi$ is stable if and only if $\psi \equiv 0$.
Moreover, since $BUC(\R)$ is a Banach module over $L^1(\R)$ under convolution, $\widehat{BUC}(\R)$ is a Banach module over the Wiener algebra $\mathcal{A}(\R) =\mathcal{F}(L^{1}(\R))$ under pointwise multiplication.
Then
	\[ \widehat{f}(T) \psi = \widehat{f} \cdot \psi , \qquad \psi \in \widehat{BUC}(\R) . \]
So, the orbit $T(t) \widehat{f}(T) \psi$ is stable if and only if $\widehat{f} \cdot \psi \equiv 0$.

We now handle $(b)$ and $(c)$ separately. 

$(b) \Longrightarrow (a)$:
Note that $\widehat{\kappa} \in \widehat{BUC}(\R)$ but $\widehat{f} \cdot \widehat{\kappa} \not\equiv 0$.
Therefore, the orbit $T(t) \widehat{f}(T) \widehat{\kappa}$ cannot be stable.
We now claim that
	\begin{equation}
		\label{eq:LocPseudoSpecIsSupport}
	 	i \sigma_{\PF}(A, \psi) = \supp \psi , \qquad \psi \in \widehat{BUC}(\R) .  
	\end{equation}
This would in particular show that $(b)$ cannot hold, since we would then have $i \sigma_{\PF}(A, \widehat{\kappa}) \subseteq E$.
Now, fix any $\psi \in \widehat{BUC}(\R)$ and $\varphi \in \D(\R)$.
A calculation shows 
	\begin{equation*}
		\| \ev{R(i \xi, A) \psi}{e^{i t \xi} \varphi(\xi)} \|_{\widehat{BUC}}
		= \Big\| \int_{-t}^\infty  \widehat{\varphi}(u)\widehat{\psi}(\cdot + u) \, \mathrm{d} u  \Big\| _{L^\infty} .
	\end{equation*}
Clearly, the right-hand side goes to zero as $t \to -\infty$.
On the other hand, 
	\[ \lim_{t \to +\infty} \| \ev{R(i \xi, A) \psi}{e^{i t \xi} \varphi(\xi)} \|_{\widehat{BUC}} = \| \widehat{\varphi}* \widehat{\check{\psi}} \|_{L^\infty} = 2\pi \|  \check{\varphi}\cdot\psi  \|_{\widehat{BUC}} . \]
 Consequently, the limit on the left-hand side vanishes if and only if $ \varphi(-\xi)\cdot \psi(\xi) \equiv 0$.
This allows us to conclude \eqref{eq:LocPseudoSpecIsSupport}.

$(c) \Longrightarrow (a)$:
We consider the Banach subspace $X$ of $\widehat{BUC}(\R)$ defined as the closure of
	\[ \mathcal{A}(\R) \cdot \widehat{\kappa} = \{ \chi \cdot \widehat{\kappa} \mid \chi \in \mathcal{A}(\R) \} . \]
Clearly $X$ is invariant under $\mathcal{T}$.
Fix $\varphi \in L^1(\R)$ with $\int_{-\infty}^{\infty}\varphi (u)\:\mathrm{d}u= 1$and define the sequence of functions $\varphi_n (u)= n \varphi(n u)$ for $n \in \N$.
Then $\{ \widehat{\varphi}_n \cdot \widehat{\kappa} \mid n \in \N \} \subseteq X$ and $\widehat{\varphi}_n \cdot \widehat{\kappa}  \to \widehat{\kappa}$ in $\widehat{BUC}(\R)$. 
This shows $\widehat{\kappa} \in X$, so that
	\[ \| T(t) \widehat{f}(T) \|_{L(X)} \geq \frac{\|T(t) \widehat{f}(T) \kappa\|_{\widehat{BUC}}}{\|\kappa\|_{L^\infty}} = \frac{\|\widehat{f} \cdot \widehat{\kappa}\|_{\widehat{BUC}}}{\|\kappa\|_{L^\infty}} > 0 . \]
Therefore, the left-hand side cannot vanish as $t \to \infty$.
As a result, the implication would follow
 if 
	\begin{equation}
		\label{eq:UniformPseudoSpecIsSupport}
	 	i \sigma_{\uPF}(A) = \supp \widehat{\kappa}  .  
	\end{equation}
As before, for any $\varphi \in \D(\R)$, we have
	\[ \lim_{t \to +\infty} \| \ev{R(i \xi, A)}{e^{i t \xi} \varphi(\xi)} \|_{L(X)} = 2\pi \sup_{\| \chi \cdot \widehat{\kappa} \|_{\widehat{BUC}} \leq 1} \| \chi \cdot \widehat{\kappa} \cdot \check{\varphi} \|_{\widehat{BUC}} , \]
while the limit as $t \to -\infty$ vanishes.
This clearly shows \eqref{eq:UniformPseudoSpecIsSupport}.
\end{proof}

\begin{remark} 
The implication $(a) \Longrightarrow (c)$ from Theorem \ref{t:KatznelsonTzafriri} was already proved in  \cite{E-S-Z-StabAsympSemGroup,V-ThmKatznelsonTzafriri} (cf. \cite[Theorem 5.2]{C-T-IdeaResult}) 
\end{remark}

We may now show Theorem \ref{t:KatznelsonTzafririIntro} from the Introduction.

\begin{proof}[Proof of Theorem \ref{t:KatznelsonTzafririIntro}]
By Theorem \ref{t:KatznelsonTzafriri}, it suffices to show that condition $(b)$ from Theorem \ref{t:KatznelsonTzafriri} is equivalent to the following statement for any bounded $C_0$-semigroup $\mathcal{T}$:
	\begin{equation}
		\label{eq:KatznelsonTzafririFull}
		i \sigma_{\PF}(A) \subseteq E \quad \Longrightarrow \quad \lim_{t \to \infty} \| T(t) \widehat{f}(T) x \|_{X} = 0 , ~\forall x \in X . 
	\end{equation}
By \eqref{eq:InclusionChainSpectra}, the latter statement is contained in $(b)$ from Theorem \ref{t:KatznelsonTzafriri}.
Suppose now that \eqref{eq:KatznelsonTzafririFull} is true for any bounded $C_0$-semigroup $\mathcal{T}$.
Take any such $\mathcal{T}$ and suppose $x_0 \in X$ is an element for which $i \sigma_{\PF}(A, x_0) \subseteq E$.
Let $Y$ be the closure of the linear span $M$ of the orbit $T(t) x_0$.
Then $\mathcal{T}_{\mid Y}$ defines a bounded $C_0$-semigroup with generator $A_{\mid Y}$.
Clearly, $\sigma_{\PF}(A_{\mid Y}, M) = \sigma_{\PF}(A, x_0)$, so also $\sigma_{\PF}(A_{\mid Y}) = \sigma_{\PF}(A, x_0)$ by Proposition \ref{PFClosure}.
In particular, $i \sigma_{\PF}(A_{\mid Y}) \subseteq E$ and then \eqref{eq:KatznelsonTzafririFull} implies that $T(t) \widehat{f}(T) x_0 \to 0$ as $t \to \infty$.
Therefore, the property $(b)$ from Theorem \ref{t:KatznelsonTzafriri} has been established.
\end{proof}

We now demonstrate that, as a matter of fact, there exist closed null sets for which Theorems \ref{t:KatznelsonTzafririIntro} and \ref{t:KatznelsonTzafriri} only yield trivial information. 

\begin{proposition}
\label{p:NoSpectralSynthesis}
There exists a closed null set $E$ such that no non-trivial element of $L^1(\R_+)$ is of spectral synthesis with respect to $E$.
\end{proposition}
\begin{proof}
We will show the stronger statement that there exists a closed null set $E$ that is a set of spectral uniqueness for $L^1(\R_+)$, meaning
	\begin{equation}
		\label{eq:SpectralUniqueness} 
		\widehat{f}(\xi)=0, \quad \forall \xi\in E
		\quad \Longrightarrow \quad f \equiv 0 , \qquad \forall f \in L^1(\R_+) . 
	\end{equation}
Since the Fourier transform of any $f \in L^1(\R_+)$ that is of spectral synthesis with respect to $E$ must necessarily vanish on $E$, this would yield our example.	

Let $\mathcal{A}^+(\mathbb{T})$ denote the analytic Wiener algebra, that is,
	\[ \mathcal{A}^+(\mathbb{T}) = \left\{ F : \mathbb{T} \to \C \mid F(\theta) = \sum_{n = 0}^\infty a_n e^{- i n \theta} \text{ with } \sum_{n = 0}^\infty |a_n| < \infty \right\} . \]
Carleson showed \cite[Theorem 8]{C-SetsUniquenessCircle} that there exists a closed null set $\Theta \subset \mathbb{T}$ that is a set of uniqueness for $\mathcal{A}^+(\mathbb{T})$, meaning
	\begin{equation}
		\label{eq:CarlesonUniqueness} 
		F_{\mid \Theta} \equiv 0 \quad \Longrightarrow \quad F \equiv 0 , \qquad \forall F \in \mathcal{A}^+(\mathbb{T}) . 
	\end{equation}
Set
	\[ E = \Theta + 2 \pi \mathbb{Z} . \]
Then $E$ is a closed null subset of $\R$.
We now verify that $E$ satisfies \eqref{eq:SpectralUniqueness}.
To do so, we follow a technique from Hedenmalm in \cite{H-Comparsionl1L1} to transfer functions from $L^1(\R_+)$ to $\mathcal{A}^+(\mathbb{T})$.
To start, we fix a countable dense subset $\Theta_0$ of $\Theta$.
Note that \eqref{eq:CarlesonUniqueness} also holds for $\Theta$ replaced with $\Theta_0$, because the elements of $\mathcal{A}^+(\mathbb{T})$  are continuous functions.

Fix any $f \in L^1(\R_+)$ for which $\widehat{f}(\xi) = 0$ for all $\xi \in E$.
We have
	\[ \sum_{n = 0}^\infty |f(n + s)| < \infty \]
for almost all $s \in (0, 1)$.
Consequently,  the functions
	\[ F_s(\theta) = \sum_{n = 0}^\infty f(n + s) e^{- i n \theta} \]
belong to $\mathcal{A}^+(\mathbb{T})$ for almost all $s \in (0, 1)$.
Take any $\theta \in \Theta_0$. 
 Since $\theta + 2\pi m \in E$ for $m \in \mathbb{Z}$, we find
	\begin{align*} 
		\int_{0}^1 e^{-2\pi i m s} \left( e^{- i \theta s} F_s(\theta) \right) \, \mathrm{d}s 
		&= \int_0^1 e^{- i (\theta + 2 \pi m) s} \left( \sum_{n = 0}^\infty f(n + s) e^{- i n \theta} \right) \, \mathrm{d}s \\
		&= \sum_{n = 0}^\infty \int_0^1 e^{- i (\theta + 2 \pi m) (n + s)} f(n + s) \, \mathrm{d}s
		= \widehat{f}(\theta + 2 \pi m) = 0 .
	\end{align*}
Therefore, the Fourier coefficients of the $L^1(0, 1)$ function
	\[ s \mapsto e^{- i \theta s} F_s(\theta) \]
are zero, so that in particular the function itself must be zero.
Hence, $F_s(\theta) = 0$ for any $\theta \in \Theta_0$ and almost all $s \in (0, 1)$, say for $s$ in $I_{\theta} \subseteq (0, 1)$, where $(0, 1) \setminus I_\theta$ is a null set.
Consequently, by \eqref{eq:CarlesonUniqueness}, we see that $F_s \equiv 0$ and $s \in \bigcap_{\theta \in \Theta_0} I_\theta$, and $(0, 1) \setminus \bigcap_{\theta \in \Theta_0} I_\theta$ is also a null set since $\Theta_0$ is countable.	
This allows us to conclude that $f(n + s) = 0$ for all $n \in \N_0$ and almost all $s \in (0, 1)$.
Necessarily $f \equiv 0$, so that $E$ satisfies \eqref{eq:SpectralUniqueness}, completing the proof of this proposition.
\end{proof}

\section{Wiener kernels and strongly stable semigroups}\label{Wiener kernels}
In connection with the previous section, we would like to point out that strong stability also admits a characterization in terms of orbit stability of the semigroup on the range of $\widehat{f}(T)$ when $f\in L^{1}(\mathbb{R}_{+})$ is a Wiener kernel. Recall \cite[Chapter II]{K-TaubTh}  $f\in L^{1}(\mathbb{R})$ is called a Wiener kernel if its Fourier transform $\widehat{f}$ never vanishes.

We need to introduce some notions in preparation for our considerations. Let $Z$ be a Banach space and let $I \subseteq \R$ be an open set. 
 We start by explaining that the elements of $\PF_{\loc}(I;Z)$ can be multiplied by more general functions than just smooth ones (see Subsection \ref{sub v-v pseudofunctions}) in a natural way. 
 Recall, we write $\mathcal{A}(\mathbb{R})=\mathcal{F}(L^{1}(\R))$ for the Wiener algebra.
We then define the local Wiener algebra $\mathcal{A}_{\loc}(I)$ as the space of continuous functions that coincide with elements of the Wiener algebra on each finite open subinterval of $I$. The inclusion $L^{1}(\R)\ast C_0(\R;Z)\subset C_0(\R;Z)$ directly transfers into an $\mathcal{A}(\R)$ multiplication module structure on $\PF(\R; Z)$. Localizing, this induces the natural module structure of  $\PF_{\loc}(I;Z)$  over the algebra  $\mathcal{A}_{\loc}(I)$.

The global space of $Z$-valued pseudomeasures is defined as $PM(\R;Z)=\mathcal{F}(L^{\infty}(\R;Z))$. Note again that $PM(\R; Z)$ is a multiplication module over the Wiener algebra $\mathcal{A}(\R)$.
We mention that for a bounded $C_{0}$-semigroup we always have $R(i\xi,A)\in PM(\R; L(X))$. 

We have the following important observation about the difference between $\sigma_{\PF}(A, x)$ and $\sigma_{\PF}(A, \widehat{f}(T) x)$.
We write $Z(\hat{f})$ for the zero set of $\widehat{f}$,
that is, $Z(\hat{f}) = \widehat{f}^{-1}(0)$.

\begin{proposition}\label{p:Wiener Tauberian semigroups} Let $\mathcal{T}$ be a bounded $C_0$-semigroup and let $f \in L^1(\R_+)$. 
Then, for each $x\in X$,
	\[ \sigma_{\PF}(A, \widehat{f}(T)x) \subseteq \sigma_{\PF}(A, x) \subseteq \sigma_{\PF}(A, \widehat{f}(T)x) \cup [- i Z(\hat{f})] . \]
\end{proposition}

\begin{proof}
Since $R(\lambda, A)$ and $T(t)$ commute for each $t\in\mathbb{R}_{+}$ and $\lambda \in \C_+$, so do $R(\lambda, A)$ and $\widehat{f}(T)$. 
Proposition \ref{l:CommutingOperator}, with $L = \widehat{f}(T)$, thus yields the inclusion $\sigma_{\PF}(A,\widehat{f}(T)x)\subseteq\sigma_{\PF}(A,x)$.

As tacitly done before, we extend $T(t) = 0$ for $t < 0$, so that $\widehat{T}(\xi) = \mathcal{F}\{T(t); \xi\} = R(i\xi, A)\in \mathcal{S}'(\R, L(X))$.
Then we find the identity
	\[ T(t) \widehat{f}(T) x = (\check{f} * [T(\cdot)x])(t) + g(t) \]
where $g(t) = 0$ for $t \geq 0$ and $g(t) = - \int_{-t}^\infty f(s) T(t + s)x \, \mathrm{d}s$ otherwise.
In particular, $g(t) = o(1)$ as $|t| \to \infty$.
Consequently, $\widehat{g} \in \PFloc(\R; X)$.
Since $\widehat{f}(-\xi) R(i \xi, A)x \in PM(\R; X)$, applying the Fourier transform on the previous identity yields
	\begin{equation} 
		\label{eq:PseudoSpecMultFourier}
		\widehat{f}(-\xi) R(i \xi, A)x  \in R(i \xi, A) \widehat{f}(T)x + \PFloc(\R; X) . 
	\end{equation}
	
Take now any open interval $I$ such that $R(i\xi, A)\widehat{f}(T)x \in \PFloc(I;X)$ and $\widehat{f}(-\xi) \neq 0$ for any $\xi \in I$.
By \eqref{eq:PseudoSpecMultFourier} we also have that $\widehat{f}(-\xi) R(i \xi, A) x \in \PFloc(I; X)$.
Fix some arbitrary $\varphi \in \D(I)$.
Since $\widehat{f}(-\xi)$ does not vanish on the support of $\varphi$, Wiener's local division theorem \cite[Theorem II.7.3, p.~81]{K-TaubTh} yields the existence of some $\psi \in \mathcal{A}(\R)\cap C_{c}(I)$ such that $\varphi(\xi)= \psi(\xi) \cdot \widehat{f}(-\xi)$.
Then
	\[ \varphi(\xi) R(i \xi, A) x = \psi(\xi) \cdot \widehat{f}(-\xi) R(i \xi, A) x \in \mathcal{A}_{\loc}(I) \cdot \PFloc(I;X) \subseteq\PFloc(I;X). \]
This shows that $R(i \xi, A) x \in \PFloc(I; X)$.
As $I$ was arbitrary, it follows that 
	\[ \sigma_{\PF}(A, x) \subseteq  \sigma_{\PF}(A, \widehat{f}(T)x) \cup [-i Z(\hat{f})] . \]
Our proof is complete.
\end{proof}

As an immediate corollary of Proposition \ref{p:Wiener Tauberian semigroups} and Theorem \ref{t:StableOrbitCountable}, when $f \in L^1(\R_+)$ is a Wiener kernel, namely, $Z(\hat{f}) = 0$,
 we obtain that the orbit $T(t)x$ is stable if and only if so is $T(t)\widehat{f}(T)x$. Moreover, we have the following characterization of Wiener kernels.

\begin{theorem}
\label{t:Wiener Tauberian semigroups}
	Let $f \in L^1(\R_+)$. Then, $f$ is a Wiener kernel if and only if for any bounded $C_0$-semigroup $\mathcal{T}$ the following implication holds:
		\[ \lim_{t \to \infty} \|T(t) y \|_X = 0, \quad \forall y\in \rg(\widehat{f}(T)) \quad \Longrightarrow \quad  \lim_{t \to \infty} \|T(t) x \|_X = 0,   \quad \forall x\in X . \]		
\end{theorem}
\begin{proof}  The necessity
 follows from Proposition \ref{p:Wiener Tauberian semigroups} and  Theorem \ref{Mainresult}.
 For the sufficiency, assume that $\widehat{f}(\xi_0)=0$ for some $\xi_0\in \R$. Consider the isometric multiplication semigroup $\mathcal{T} = (e^{-i\xi_0t})_{t\geq0}$ on $X=\C$. Clearly, $\rg (\widehat{f}(T))=\{0\}$ but the semigroup is not strongly stable.
\end{proof}

\section{Application: asymptotically almost periodic semigroups}
\label{sec:AlmostPeriodic}

We shall now apply our findings to obtain a characterization of asymptotically almost periodic semigroups \cite[p.~365]{A-B-H-N-VVLaplaceTransCauchyProb} with countable local pseudofunction spectrum, thereby generalizing and improving a result from \cite{B-V-R-LocalSpecInStab}.
In \cite{B-V-R-LocalSpecInStab}, sufficient conditions for asymptotic almost periodicity were given in terms of the local unitary spectrum; here, we show that using the local pseudofunction spectrum already suffices (which, in general, can be much smaller, see Examples \ref{ex:PFEmptyButNotContinuous} and \ref{ex:PseudoZeroUnitaryWholeLine}).

\begin{definition}
A $C_0$-semigroup $\mathcal{T}$ is called \emph{(weakly) asymptotically almost periodic}\footnote{Some authors simply call these semigroups almost periodic \cite{V-L-SpecCritAlmostPeriodic}.} if for each $x \in X$ the orbit $\{ T(t) x \mid t \geq 0 \}$ is relatively (weakly) compact.
\end{definition}

By Mackey's theorem, every weakly asymptotically almost periodic $C_0$-semigroup must be bounded. 

Note that if the orbit $T(t)x$ is stable, it is totally bounded, and thus relatively compact.
Stable semigroups are therefore asymptotically almost periodic, but the converse is not true in general, which can be seen from the following trivial example. 

\begin{example}
	Consider the multiplication semigroup $\mathcal{T} = (e^{it})_{t\geq0}$ on $\C$, which is isometric, hence not strongly stable. 
	However, for each $x \in \C$, its orbit $\{e^{it} x \mid t\geq0\}$ is relatively compact, hence $\mathcal{T}$ is asymptotically almost periodic.
\end{example}

We will actually work with individual orbits. The next result improves upon \cite[Theorem 6.1]{B-V-R-LocalSpecInStab}, where a characterization was obtained when $\sigma_u(A, x)$ is countable. 

\begin{theorem}
\label{t:AlmostPeriodic}
    Let $\mathcal{T}$ be a bounded $C_0$-semigroup. For any $x \in X$ with $\sigma_{\PF}(A, x)$ countable, the following properties are equivalent:
    \begin{enumerate}
        \item[$(a)$] The orbit $\{T(t)x\mid t\geq0\}$ is relatively compact;
        \item[$(b)$] the orbit $\{T(t)x\mid t\geq0\}$ is relatively weakly compact;
        \item[$(c)$] $\lim_{\alpha\to0+}\alpha R(\alpha+ i\xi, A)x$ exists for all $i \xi \in \sigma_{\PF}(A, x)$;
        \item[$(d)$] $\wlim{\alpha\to0+}\alpha R(\alpha + i\xi, A)x$ exists for all $i \xi \in \sigma_{\PF}(A, x)$.
    \end{enumerate}
\end{theorem}
\begin{proof}

The implications $(a) \Longrightarrow (b)$ and $(c) \Longrightarrow (d)$ are trivial, while $(b) \Longrightarrow (c)$ follows from \cite[Theorem 2.1.5, p.~76]{K-ErgodicThm}. It just remains to establish $(d) \Longrightarrow (a)$. For it, we follow the same proof method as in \cite[Theorem 6.1]{B-V-R-LocalSpecInStab}, supplemented by the new crucial ingredient Theorem \ref{t:StableOrbitCountable}.

Possibly renorming $X$, we may assume that $\mathcal{T}$ is a contraction semigroup.
We consider the subspace
	\[ X_{\text{ap}} = \{ y \in X \mid \text{the orbit } \{ T(t) y \mid t \geq 0 \} \text{ is relatively compact} \} . \]
Clearly, $X_{\text{ap}}$ is closed in $X$ and invariant under $\mathcal{T}$. Hence $\mathcal{T}$ induces a bounded $C_0$-semigroup $\widetilde{\mathcal{T}} = (\widetilde{T}(t))_{t \geq 0}$ on $X / X_{\text{ap}}$ for which $\widetilde{T}(t) \circ q = q \circ T(t)$ for $t \geq 0$, where $ q : X \to X / X_{\text{ap}} $ is the quotient map. Observe then that
	\[ \widetilde{T}(t) q(x) \to 0 \quad \text{as } t\to\infty \quad {\iff}
	 \quad \{ T(t)x \mid t \geq 0 \} \text{ is relatively compact} . \]
Indeed, take any $\varepsilon > 0$ and suppose $t_0 \geq 0$ is such that $\|\widetilde{T}(t_0) q(x)\|_{X / X_{\text{ap}}} \leq \varepsilon / 2$.
Hence, there exists some $y \in X_{\text{ap}}$ such that $\|T(t_0) x - y \|_X \leq \varepsilon$.
Then, also $\| T(t + t_0) x - T(t) y \|_X \leq \varepsilon$ for $t \geq 0$.
Consequently, we have $\{ T(t) x \mid t \geq 0 \} \subseteq K_\varepsilon + \varepsilon \overline{B}_{X}(0, 1)$ with $K_\varepsilon =  \{T(t)x\mid t\in [0,t_0]\} \cup \overline{\{T(t)y\mid t\geq0\}}$ compact.
As $\varepsilon$ was chosen arbitrarily, we see that $\{ T(t) x \mid t \geq 0 \}$ is totally bounded, therefore relatively compact.

It now suffices to show that the orbit $\widetilde{T}(t) q(x)$ is stable in the Banach space $X / X_{\text{ap}}$.
Let $\widetilde{A}$ denote the generator of $\widetilde{\mathcal{T}}$.
We have $R(\lambda, \widetilde{A}) \circ q = q \circ R(\lambda, A)$ for $\lambda \in \C_+$. Thus
	\[\sigma_{\PF}(\widetilde{A}, q(x)) \subseteq \sigma_{\PF}(A, x).\]
In particular, $\sigma_{\PF}(\widetilde{A}, q(x))$ is countable.
Take now any arbitrary $i \xi \in \sigma_{\PF}(A, x)$ and set
	\[ x_{\xi} = \wlim{\alpha \to 0^+} \alpha R(\alpha + i \xi, A) x = \wlim{\alpha \to 0^+} \alpha \int_{0}^\infty e^{- (\alpha + i \xi) t}T(t) x  \, \mathrm{d}t , \]
which exists by our assumption $(d)$. 
Then $T(t) x_\xi = e^{i \xi t} x_\xi$ for $t \geq 0$, so $x_\xi \in X_{\text{ap}}$.
Since $q$ is weakly continuous, we get
	\[  \wlim{\alpha \to 0+} \alpha R(\alpha + i \xi, \widetilde{A}) q(x) = q\left( \wlim{\alpha \to 0+} \alpha R(\alpha + i \xi, A) x \right) = q(x_\xi) = 0 .  \]
By Theorem \ref{t:StableOrbitCountable}, we see that the orbit $\widetilde{T}(t) q(x)$ is stable, and we are done.
\end{proof}

\begin{corollary}\label{c:almostperiodic} Let $\mathcal{T}$ be a bounded $C_0$-semigroup such that $\sigma_{\PF}(A)$ is countable. Then, $\mathcal{T}$ is asymptotically almost periodic if and only if there is a dense subspace $M\subseteq X$ such that $\lim_{\alpha\to0+}\alpha R(\alpha + i\xi, A)x$ exists for all $i \xi \in \sigma_{\PF}(A)$ and $x\in M$.
\end{corollary}
\begin{proof} It is enough to prove the converse. The assumption actually implies that $\lim_{\alpha\to0+}\alpha R(\alpha + i\xi, A)x$ exists for all $i \xi \in \sigma_{\PF}(A)$ and $x\in X$, in view of the density of $M$ and the resolvent bound \eqref{eq:resolvent bound}. Hence Theorem \ref{t:AlmostPeriodic} yields the result.
\end{proof}
Note that the Lyubich-V\~{u} criterion for asymptotic almost periodicity \cite{V-L-SpecCritAlmostPeriodic} is a particular instance of Corollary \ref{c:almostperiodic}.

\appendix

\section{The proof of Theorem \ref{t:SuffCondVanishSingSupp}}
\label{appendix:ProofTheorem}
In order to keep the article self-contained, we give here a proof of Theorem \ref{t:SuffCondVanishSingSupp}. Our arguments are essentially the same as those from \cite{D-V-ComplexThmLaplaceLocalPseudo}, except for a simplification where we replace the use of Romanovski's lemma \cite{E-V-Romanovski} by a contradiction argument. We start with a vector-valued version of \cite[Lemma 4.6]{D-V-ComplexThmLaplaceLocalPseudo}.

\begin{lemma}
\label{AuxiliaryLemma}
Let $\tau\in L^{\infty}(\R; Z)$ and let $J \subseteq \R$ be open. Suppose
	\[ \widehat{\tau} \in \PFloc\left( J \setminus \bigcup_{j=0}^n \left[ \xi_j - \frac{l_j}{2}, \xi_j + \frac{l_j}{2} \right] ; Z\right) , \]
where $[\xi_j-l_j, \xi_j+l_j] \subseteq J$ are disjoint. 
There is an absolute constant $C > 0$ such that
    \[
        \limsup_{|t| \to\infty}\Vert \langle \widehat{\tau}(\xi), e^{it\xi} \varphi(\xi)\rangle\Vert_Z \leq C Q \Vert \widehat{\varphi}\Vert_{L^1(\R)} \sum^n_{j=0}l_j, \qquad \varphi \in \D(J) ,
    \]
with $Q = \max_{j=0,\ldots,n}\sup_{t \in \R} \Vert \int^t_0 e^{-i \xi_j s}\tau(s) \,\mathrm{d}s\Vert_Z$.
\end{lemma}
\begin{proof}[Proof of Lemma \ref{AuxiliaryLemma}] We simply reproduce here the proof of \cite[Lemma 4.6]{D-V-ComplexThmLaplaceLocalPseudo}.
We may assume $Q < \infty$; otherwise, the statement is trivial.
Let $\chi \in \D((-1, 1))$ be even such that $\chi(\xi) \equiv 1$ in a neighborhood of $[-1/2, 1/2]$. 
Set $\chi_j(\xi) = \chi((\xi - \xi_j) / l_j)$ for $j \in \{0, \ldots, n\}$. Recall $\check{\tau}(t) = \tau(-t)$.
For any $\varphi \in \D(J)$, our assumption yields that $\widehat{\tau}$ has local pseudofunction behavior on a neighborhood of the support of $(1 - \sum_{j=0}^n \chi_j) \varphi$.
Then, 
	\begin{align*}
		\limsup_{|t| \to \infty} \Vert \langle \widehat{\tau}(\xi), e^{it\xi}\varphi(\xi)\rangle\Vert_Z
		&\leq \limsup_{|t| \to \infty} \sum_{j = 0}^n \Vert \langle \widehat{\tau}(\xi), \chi_j(\xi) e^{it\xi}\varphi(\xi)\rangle\Vert_Z \\
		&\leq \frac{\|\widehat{\varphi}\|_{L^1(\R)} }{2\pi}\sum_{j = 0}^n \sup_{r \in \R} \| (\widehat{\chi}_j*\check{\tau})  (r) \|_Z \\
		&\leq \frac{\|\widehat{\varphi}\|_{L^1(\R)} }{2\pi}\sum_{j = 0}^n l_j \sup_{r \in \R} \Big\| \int_{-\infty}^{\infty} \widehat{\chi}(l_j (r + s))e^{- i \xi_j s} \tau(s)  \, \mathrm{d}s \Big\|_Z  \\
		&\leq \frac{\|\widehat{\chi}^\prime\|_{L^1(\R)}}{2 \pi} Q \|\widehat{\varphi}\|_{L^1(\R)}\sum_{j = 0}^n l_j ,
	\end{align*}
	where we used integration by parts in the last step.
\end{proof}

We can now show Theorem \ref{t:SuffCondVanishSingSupp}. 
\begin{proof}[Proof of Theorem \ref{t:SuffCondVanishSingSupp}] If $\widehat{\varphi}\in\mathcal{D}(I)$, one readily checks that  $\singsupp_{\PF} \widehat{\varphi} \cdot \widehat{\tau}\subseteq E$ and
	\[ \sup_{t \in \R} \Big\| \int_0^t e^{- i \xi s} (\varphi\ast\tau)(s) \, \mathrm{d}s \Big \|_Z < \infty, \]
for each $\xi\in E$. Arguing via localizations, we may therefore assume without loss of generality that $\supp \widehat{\tau}$ is a compact, $I=\R$, and that $E=\singsupp_{\PF} \widehat{\tau}$ is a compact null set. We must show that $\singsupp_{\PF} \widehat{\tau}=\varnothing$.

We reason by contradiction. So, assume that $\singsupp_{\PF} \widehat{\tau}\neq \varnothing.$ Considering the closed subsets

\[ \left\{ \xi \in \singsupp_{\PF} \widehat{\tau} \mid \sup_{t \in\R} \Big\Vert \int^t_0 e^{-i \xi s} \tau(s) \, \mathrm{d}s \Big\Vert_Z \leq Q \right\} , \qquad Q \in \N , \]
the Baire category theorem guarantees the existence of some finite open  interval $J \subseteq \R$ and some $Q\in \N$ such that $\varnothing \neq E_{0}=J\cap \singsupp_{\PF} \widehat{\tau}$ and  
\[
\sup_{\xi \in E_0}\sup_{t \in\R} \Big\Vert \int^t_0 e^{-i \xi s} \tau(s) \, \mathrm{d}s \Big\Vert_Z \leq Q.
\]
Reducing the size of $J$ if necessary, we may assume that $E_{0}$ is compact in addition to having zero Lebesgue measure. Fix $\varepsilon>0$. There is then a finite covering $E_0 \subseteq \bigcup_{j = 0}^n [\xi_j - \ell_j / 2, \xi_j + \ell_j / 2]$, for certain $\xi_j \in E_0$, such that $\sum_{j=0}^n \ell_j < \varepsilon$ and the intervals $[\xi_j - \ell_j, \xi_j + \ell_j] \subseteq J$ are disjoint.
Consequently, $\tau$ satisfies the hypotheses from Lemma \ref{AuxiliaryLemma}, which yields, for some absolute constant $C > 0$,
	\[  \limsup_{\vert t\vert\to\infty}\Vert \langle \widehat{\tau}(\xi), e^{it\xi}  \varphi(\xi)\rangle\Vert_Z \leq \varepsilon C Q \Vert \widehat{\varphi}\Vert_{L^1(\R)}, \]
for any $\varphi \in \D(J)$. Letting $\varepsilon \to 0^+$, the arbitrariness of $\varphi$ allows us to conclude that $\tau\in\PFloc(J)$, or equivalently, that $E_{0}=\varnothing$, producing a contradiction.
\end{proof}

\end{document}